\def\LaTeX{%
  \let\Begin\begin
  \let\End\end
  \let\salta\relax
  \let\finqui\relax
  \let\futuro\relax}

\LaTeX

\relax

\documentclass[twoside,11pt]{article}
\usepackage[usenames,dvipsnames]{color}
\usepackage{amsmath}
\usepackage{amsthm}
\usepackage{amssymb}
\usepackage[mathcal]{euscript}
\usepackage{cite}
\usepackage{hyperref}
\usepackage{enumitem}
\definecolor{rosso}{rgb}{0.8,0,0}
\definecolor{darkgreen}{rgb}{0,0.5,0}

\newtheorem{theorem}{Theorem}[section]

\newtheorem{corollary}[theorem]{Corollary}
\newtheorem{definition}[theorem]{Definition}

\newtheorem{lemma}[theorem]{Lemma}

\relax

\let\non\nonumber

\def\Lip{Lip\-schitz}
\def\Holder{H\"older}
\def\Frechet{Fr\'echet}

\def\rhs{right-hand side}
\def\sfw{straightforward}

\def\pto{.}

\def\multical #1{\def\arg{#1}%
  \ifx\arg\pto \let\next\relax
  \else
  \def\next{\expandafter
    \def\csname cal#1\endcsname{{\cal #1}}%
    \multical}%
  \fi \next}

\def\multimathop #1 {\def\arg{#1}%
  \ifx\arg\pto \let\next\relax
  \else
  \def\next{\expandafter
    \def\csname #1\endcsname{\mathop{\rm #1}\nolimits}%
    \multimathop}%
  \fi \next}

\multical
QWERTYUIOPASDFGHJKLZXCVBNM.

\multimathop
ad dist div dom meas sign supp .

\def\Accorpa #1#2 #3 {\gdef #1{\eqref{#2}--\eqref{#3}}%
  \wlog{}\wlog{\string #1 -> #2 - #3}\wlog{}}

\def\Accorpadue #1#2 #3 {\gdef #1{\ref{#2}--\ref{#3}}%
  \wlog{}\wlog{\string #1 -> #2 - #3}\wlog{}}

\def\<#1>{\mathopen\langle #1\mathclose\rangle}
\def\norma #1{\mathopen \| #1\mathclose \|}

\def\iot {\int_0^t}

\def\intQt{\int_{Q_t}}
\def\intQ{\int_Q}
\def\iO{\int_\Omega}
\def\intQtT{\int_{Q_t^T}}

\def\dt{\partial_t}
\def\dn{\partial_{\bold {n}}}

\def\checkmmode #1{\relax\ifmmode\hbox{#1}\else{#1}\fi}

\def\erre{{\mathbb{R}}}

\def\genspazio #1#2#3#4#5{#1^{#2}(#5,#4;#3)}
\def\spazio #1#2#3{\genspazio {#1}{#2}{#3}T0}

\def\L {\spazio L}
\def\H {\spazio H}
\def\W {\spazio W}

\def\C #1#2{C^{#1}([0,T];#2)}

\def\Lx #1{L^{#1}(\Omega)}
\def\Hx #1{H^{#1}(\Omega)}

\let\theta\vartheta
\let\eps\varepsilon
\let\lam\lambda

\let\TeXchi\chi                         
\newbox\chibox
\setbox0 \hbox{\mathsurround0pt $\TeXchi$}
\setbox\chibox \hbox{\raise\dp0 \box 0 }
\def\chi{\copy\chibox}

\def\bQ{\beta_Q}
\def\bO{\beta_\Omega}
\def\phQ{\ph_Q}
\def\phO{\ph_\Omega}
\def\Uad{\calU_{\ad}}
\def\Vp{V^*}
\def\cd{M(\d)}
\def\s{\sigma}  
\def\m{\mu}	    
\def\ph{\varphi}	
\def\d{\delta}  
\def\r{\rho}    
\def\bph{\overline\ph}  
\def\bm{\overline\m}    
\def\bs{\overline\s}    
\def\z{\zeta}      
\def\ch{\chi}      
\def\J{{\cal J}}		
\def\Jred{{\cal J}_{\rm red}} 
\def\Jad{{\cal J}_{\rm ad}} 
\def\S{{\cal S}}   
\def\I2 #1{\int_{Q_t}|{#1}|^2}
\def\IO2 #1{\norma{{#1(t)}}^2}

\def\ov #1{{\overline{#1}}}
\def\bh{\bold h}
\def\aP{\alpha_{\P}}
\def\achi{\alpha_{\chi}}
\def\aeta{\alpha_{\eta}}
\def\aC{\alpha_{\CC}}
\def\bhph{\bph^{\bh}}
\def\bhmu{\bm^{\bh}}
\def\bhs{\bs^{\bh}}
\def\ept{{\eps,\tau}}
\def\X{{\cal X}}
\def\phdiff{\bhph-\bph}
\def\sdiff{\bhs-\bs}

\def\opt{(\ov{P}, \ov{\chi},\ov{\eta},\ov{C})}

\def\redopttau{(\ov{\P}_\tau, \ov{\chi}_\tau,\ov{\CC}_\tau)}

\def\redoptept{(\ov{\P}_\ept, \ov{\chi}_\ept,\ov{\CC}_\ept)}

\def\CPad{(CP)^{\rm ad}_\ept}
\def\CPept{(CP)_\ept}
\def\CPeps{(CP)_\eps}
\def\CPtau{(CP)_\tau}
\def\CPoo{({\ov {CP}})}
\def\x{{\bf x}}
\def\dx{{\rm d {\x}}}
\def\y{{\bf y}}
\def\dy{{\rm d {\y}}}
\def\ds{{\rm ds}}
\def\A{{\mathscr A}}

\def\B{{\mathscr B}}
\def\P{{\mathscr P}}
\def\CC{{\mathscr{C}}}

\usepackage{amsmath}
\DeclareFontFamily{U}{mathc}{}
\DeclareFontShape{U}{mathc}{m}{it}%
{<->s*[1.03] mathc10}{}

\DeclareMathAlphabet{\mathscr}{U}{mathc}{m}{it}

\begin{document}

\title{Parameter identification for nonlocal phase field models for
tumor growth
via optimal control and asymptotic analysis}
\author{}
\date{}
\maketitle

\begin{center}
\vskip-1cm
{\large\sc Elisabetta Rocca$^{(1)}$}\\
{\normalsize e-mail: {\tt elisabetta.rocca@unipv.it}}\\[.25cm]
{\large\sc Luca Scarpa$^{(2)}$}\\
{\normalsize e-mail: {\tt {luca.scarpa@univie.ac.at}}}\\[.25cm]
{\large\sc Andrea Signori$^{(3)}$}\\
{\normalsize e-mail: {\tt andrea.signori02@universitadipavia.it}}\\[.25cm]
$^{(1)}$
{\small  Universit\`a degli studi di Pavia, Dipartimento di Matematica,
and IMATI-C.N.R}\\
{\small via Ferrata 5, 27100 Pavia, Italy}\\[.2cm]
$^{(2)}$
{\small Faculty of Mathematics, University of Vienna}\\
{\small Oskar-Morgenstern-Platz 1, 1090 Vienna, Austria}
\\ [.2cm]
$^{(3)}$
{\small Dipartimento di Matematica e Applicazioni, Universit\`a di Milano--Bicocca}\\
{\small via Cozzi 55, 20125 Milano, Italy}

\end{center}
\Begin{abstract}\noindent
We introduce the problem of parameter identification for 
a coupled nonlocal 
Cahn-Hilliard-reaction-diffusion PDE system stemming from  a  
recently introduced tumor growth model. 
The inverse problem of identifying
relevant parameters is studied here by relying on techniques from
optimal control theory of PDE systems. 
The parameters to be identified play the role of controls, and
a suitable cost functional of tracking-type is introduced
to account for the discrepancy between some {\it a priori} knowledge
of the parameters and the controls themselves.
The analysis is carried out for several classes of models, 
each one depending on a specific relaxation (of parabolic or viscous type)
performed on the original system.
First-order necessary optimality conditions are obtained 
on the fully relaxed system, in both the two and three-dimensional case.
Then, the optimal control problem on
the non-relaxed models is tackled by means of asymptotic arguments,
by showing convergence of the respective adjoint systems and 
the minimization problems as each one of 
the relaxing coefficients vanishes.
This allows obtaining
the desired necessary optimality conditions, hence to
solve the parameter identification problem, for the original PDE system 
in case of physically relevant double-well potentials.
\vskip3mm


\vskip3mm


\End{abstract}

\pagestyle{myheadings}
\newcommand\testopari{\sc Rocca -- Scarpa -- Signori}
\newcommand\testodispari{}
\markboth{\testopari}{}

\relax
\relax

\section{Introduction}

One of the main examples of complex systems studied nowadays in both the biomedical and 
the mathematical literature refers to tumor growth processes.
In particular, there has been a recent surge in the development of 
phase field models for tumor growth.  
These models are one of the main examples of complex systems: they
describe the evolution of a tumor mass surrounded by healthy tissues 
by taking into account biological mechanisms such as proliferation of cells 
via nutrient consumption, apoptosis, chemotaxis, and active transport of specific chemical 
species. In this setting, the evolution of
the tumor is described by means of an order parameter $\ph$ which represents the 
local concentration of tumoral cells.
The interface between the tumoral and healthy cells 
is considered to be a (narrow) layer separating
the regions where $\ph=\pm1$, with $\ph=1$ being the tumor phase
and $\ph= -1$ being the healthy phase. 

The representation of the tumor
growth process is given here by a Cahn-Hilliard equation (cf.~\cite{cahn-hill}) with non zero mass source for $\ph$,
coupled with a reaction-diffusion equation for the nutrient $\sigma$
(cf., e.g., \cite{crist, garcke, hawk2, hil-zee}). We  just mention the fact that more sophisticated models may distinguish
between different tumor phases (e.g., proliferating and necrotic), or, treating
the cells as inertia-less fluids, include the effects of fluid flow into the evolution 
of the tumor, leading to (possibly multiphase) Cahn-Hilliard-Darcy or 
Cahn-Hilliard-Brinkman systems
\cite{CL,Dai, frig-lam-rocca-schim, garcke,WLFC,eben-garcke1,eben-garcke2,eben-lam,garcke-lam,garcke-lam3,garcke-lam4}.

In this paper we study the inverse problem of parameters identification for a nonlocal version of a recently introduced phase field model of tumor growth (cf.~\cite{garcke, SS}). 
The nonlocality of the system consists in the presence of 
a convolution term in the Cahn-Hilliard equation (see
\eqref{state:2} below),
which comes from the variational derivative of the following nonocal Helmholtz free energy functional (cf.~\cite{GL, GL1, GL2})
\begin{equation}
\label{energy:nonloc}
\mathcal{E}_{nonloc}(\ph):=\frac14\int_{\Omega\times\Omega} J({\bf x}-{\bf y})
|\ph({\bf x})-\ph({\bf y})|^2\, \dx\, \dy+\int_\Omega F(\ph({\bf x}))\, \dx\,.
\end{equation}
This replaces the standard local Ginzburg-Landau free energy 
\begin{equation}
\label{energy:loc}
{\mathcal E}_{loc}(\ph):=\frac12\int_{\Omega} 
|\nabla\ph({\bf x})|^2\, \dx+\int_\Omega F(\ph({\bf x}))\, \dx\,.
\end{equation}
Here $J$ stands for a spatial convolution kernel, such as the classical Bessel or Newtonian potentials, while typical choices for the non-convex interaction potentials $F$ are
\begin{align}
  \label{Fpol}
  &F_{pol}(r):=\frac14(r^2-1)^2\,, \quad r\in\erre\,,\\
  \label{Flog}
  &F_{log}(r):=\frac{\theta}{2}\left[(1+r)\ln(1+r)+(1-r)\ln(1-r)\right]-\frac{\theta_0}{2}r^2\,, \quad r\in(-1,1)\,,
  \quad0<\theta<\theta_0\,,\\
  \label{Fdobs}
  &F_{dob}(r):=\begin{cases}
  c(1-r^2) \quad&\text{if } r\in[-1,1]\,,\\
  +\infty \quad&\text{otherwise}\,,
  \end{cases}
  \qquad c>0\,.
\end{align}
Let us refer to \cite{gal-gior-grass} and reference therein for a description of the state of the art on nonlocal Cahn-Hilliard equations and to the recent contributions 
\cite{DST, DST2, DRST, MRT18} for the results concerning the rigorous passage from nonlocal to local Cahn-Hillliard equations when the kernel suitably peaks around zero. 

Here, $\Omega\subset\erre^d$ ($d=2,3$) is the space domain where the evolution 
takes place, and $T>0$ is a fixed final time. We consider 
the following nonlocal state system, deduced 
starting from 
the nonlocal free energy functional \eqref{energy:nonloc},
which has been recently studied from the well-posedness point of view in \cite{SS}:
\begin{align}
	\label{state:1}
	&\dt(\eps  \m
	+ \ph) - \Delta \m
	\,=\,
	(\P\s -\A) f(\ph)
	 &&\hbox{in $\, Q:=\Omega \times (0,T),$}
	\\[0.5mm]
	\label{state:2}
	&\m
	\,=\,
	\tau\dt \ph
	+a \ph -J * \ph
	+ F'(\ph)
	- \ch \s
 	&&\hbox{in $\,Q,$}
	\\[0.5mm]
	\label{state:3}
	&\dt \s
	- \Delta \s
	+\B(\s - \s_S)
	+\CC \s f(\ph)
	\,=\,
	- \eta \Delta \ph
	 &&\hbox{in $\,Q,$}
	\\[0.5mm]
	&\dn \m\,=\,\eta \dn \ph\,=\,\dn \s \,=\,0
 	 &&\hbox{on $\,\Sigma:=\partial\Omega \times (0,T),$}
	\label{state:4}
	\\[0.5mm] 
	&\m(0)\,=\, \m_0, \ \ph(0)\,=\,\ph_0, \ \s(0)\,=\,\s_0
  	 &&\hbox{in $\,\Omega$.}
	\label{state:5}
\end{align}
\Accorpa\statesys {state:1} {state:5}
Referring to \cite{garcke} for more details on the physical meaning of the involved parameters, let us precise our notation here. The variable $\sigma$ denotes the
the concentration of the surrounding nutrient, 
where $\sigma=1$ stands for a nutrient-rich concentration while $\sigma=0$ for a poor one.
The variable $\mu$ represents the chemical potential associated 
to the tumor phase-field variable $\ph$ and to the nutrient proportion $\sigma$.
Moreover, we use the symbols ${\bf n}$ and $\dn$ for the outer normal unit vector to 
$\partial \Omega$ and the outer normal derivative, respectively. 
The constants $\P, \A, \B, \CC$ are nonnegative real parameters taking into account the proliferation rate of the tumor cells, the apoptosis rate, 
the consumption rate of the nutrient with respect to a pre-existing concentration $\sigma_S$, 
and the nutrient consumption rate, respectively. 
Moreover, the nonnegative constants $\chi$ and $\eta$ 
model chemotaxis and active transport, respectively. 
The approximation coefficients $\tau$ and $\eps$ represent a viscosity and a parabolic regularization of the Cahn-Hilliard equation, respectively.

 In biological models, nonlocal interactions have been used to describe competition for space and degradation \cite{SRLC}, spatial redistribution \cite{Borsi, Lee}, and also cell-to-cell adhesion \cite{Armstrong, Chaplin, Gerisch}.   The model \statesys which we study falls roughly to the category of nonlocal cell-to-cell adhesion, as it is well-known that the Ginzburg-Landau energy leads to separation and surface tension effects, and heuristically this corresponds to the preference of tumor cells to adhere to each other rather than to the healthy cells.

For the nonlocal Cahn-Hilliard equation with source terms, analytic results such as well-posedness and long-time behavior have been obtained in \cite{DellaPorta2, Melchionna} for prescribed source terms or Lipschitz source terms depending on the order parameter.  The aim of \cite{SS} was to extend the study of the nonlocal Cahn-Hilliard equation to the case where source terms are coupled with another variable: the nutrient concentration $\sigma$. In \cite{SS} the authors prove  well-posedness and regularity of solutions of \statesys  when $\tau$ and $\eps$ are positive, moreover they let them tend to zero and prove existence of solutions of the corresponding limit problem. 

The contributions to the mathematical literature on local diffuse interface tumor growth models is quite wide (cf, e.g., \cite{CRW, col-gil-hil, col-gil-mar, col-gil-spr-fraccontrol, col-gil-roc-spr, col-gil-roc-spr2,col-gil-roc-spr-control,col-signori-spre, garcke, garcke-lam, garcke-lam2,eben-garcke1, eben-garcke2, garcke-lam-roc, garcke-lam-signori, frig-grass-roc, hil-zee, kahle-lam, MRS, signori1,signori2,signori3,signori4,signori5}) and it is devoted to different mathematical questions: not only are these related to well-posedness, but they also focus on
long-time behavior of solutions and optimal control, 
possibly including more accurate effects in the model, like stochastic ones \cite{orr-roc-scar}.
Nonetheless, the study of the nonlocal corresponding version is still at its beginning:
we can indeed quote only the few very recent contributions \cite{frig-lam-roc, frig-lam-signori, SS, FLOW, fritzlocal}.

Here we aim to continue the study pursued in \cite{SS} investigating the inverse problem of parameter identification in \statesys\ by means of optimal control theory.
 In order to do this we need to introduce a cost functional we are going to minimize.
This technique has been already used in the framework of diffuse interface tumor growth models in \cite{frig-lam-signori, kahle-lam, kahle-lam-latz}, where the problem of parameters identification has been tackled with similar methods for the corresponding local tumor growth model  
(in \cite{kahle-lam, kahle-lam-latz})  and for a nonlocal model introduced in \cite{hawk2} (in \cite{frig-lam-signori}). 

The parameters that we choose to identify are $\P$, $\ch$, $\eta$, and $\CC$.
The main idea is to introduce a {\it cost functional} of standard tracking-type form  as
\begin{align}
	\non
	\J(\ph,\P,\ch,\eta,\CC) &:= 
		\frac {\bO} 2 \norma{\ph(T)-\phO}_{L^2(\Omega)}^2		
		+ \frac {\bQ} 2 \norma{\ph-\phQ}_{L^2(Q)}^2	
		\\ & \quad 
		+\frac {\aP}2 |\P-\P_*|^2
		+\frac {\achi}2 |\ch-\ch_*|^2
		+\frac {\aeta}2 |\eta-\eta_*|^2
		+\frac {\aC}2 |\CC-\CC_*|^2,
		\label{cost}
\end{align}
for some prescribed target functions $\phO:\Omega \to \erre, \, \phQ:Q \to \erre$,
and some nonnegative weights $\bO,\bQ,\aP,\achi,\aeta,\aC$ (not all zero).
The nonnegative constants $\P_*,\chi_*,\eta_*,\CC_*$ represent instead
some a priori knowledge for the parameters.
Moreover, the set of {\it admissible controls} is defined as
\begin{align}
	 \non
	\Uad:= &\{(\P,\chi,\eta,\CC) \in \erre^4: 
		\\
		\label{Uad}
		& \qquad
		0\leq \P \leq \P_{\rm max}, \;
		0\leq \chi \leq \chi_{\rm max},\;
		0\leq \eta \leq \eta_{\rm max},\;
		0\leq \CC \leq \CC_{\rm max} \},
\end{align}
for some prescribed nonnegative constants 
$\P_{\rm max},\chi_{\rm max},\eta_{\rm max},\CC_{\rm max}$.
It is worth underlying that \eqref{Uad} is a nonempty compact (closed and bounded) 
subset of $\erre^4$.

The identification problem we are going to address in this work 
can be summarized as 
\begin{align*}
	(CP)_\ept \quad
	&\hbox{Minimize $\J(\ph,\P,\ch,\eta,\CC)$ subject to:}\\
	&\hbox{(i) $(\ph,\mu,\s)$ yields 
	a solution to \statesys;}
	\\ 
	& \hbox{(ii) $(\P,\chi,\eta,\CC) \in \Uad.$ }
\end{align*}
More specifically, we would like to tackle the following problem: 
given a set of data $\P_*,\chi_*,\eta_*,\CC_*$ representing some a priori 
knowledge of the parameters, 
identify the optimal parameter values $\P, \ch, \eta, \CC$ so that the resulting 
model predictions and the data are close in the sense defined in
the cost functional \eqref{cost}.
An alternative approach for parameter identification relies on the 
Bayeasian calibration which has been recently used in 
\cite{kahle-lam-latz} in the framework of local tumor growth models.

The second part of the paper concerns the asymptotic behavior of $(CP)_\ept$
as $\eps$ and/or $\tau$ approach zero in the state system above. 
In this direction, we will employ the symbols 
$\CPeps$, $\CPtau$, and $(\ov{CP})$ 
for the corresponding optimal control problems in which 
$\tau=0$, $\eps=0$ and $\eps=\tau=0$ in the order.
For instance, we have
\begin{align*}
	\CPeps \quad
	&\hbox{Minimize $\J(\ph,\P,\ch,\CC)$ subject to:}\\
	&\hbox{(i) $(\ph,\mu,\s)$ yields 
	a solution to \statesys\ with $\tau=0$;}
	\\ 
	& \hbox{(ii) $(\P,\chi,\eta,\CC) \in \Uad.$ }
\end{align*}
The problems $\CPtau$ and $(\ov{CP})$ are defined analogously. 
Different requirements on the structural data are in order,
depending on the asymptotic study under consideration.
Precise assumptions are rigorously stated in the sections below.

The main idea to solve the limit identification problems 
$\CPeps$, $\CPtau$, and $(\ov{CP})$ is the following. 
We know from \cite{SS} that the fully relaxed state system \statesys\ converges to 
the corresponding limiting one, as $\tau$ and/or $\eps$ vanishes.
Based on this, we introduce a suitable approximating cost functional 
in such a way that also the respective adjoint systems converge
in some sense with respect to $\tau$ and/or $\eps$. This allows 
to pass to the limit in the first-order conditions for optimality 
by letting $\tau$ and/or $\eps$ go to $0$, and deduce the 
corresponding first-order conditions also for the limiting models.
The main mathematical issue concerns
the nonuniqueness of optimal controls: this is a consequence of the 
highly nonlinear nature of the problem, and may give some 
difficulties in the convergence of optimal controls.
Indeed, it not necessarily true that optimal controls at $\eps,\tau>0$
converge to optimal controls at $\tau=0$ and/or $\eps=0$. 
To overcome this issue, the cost functional at the approximated level has 
to be carefully adapted in order to recover, among the other things, 
convergence of the optimal controls.

\noindent
{\bf Plan of the paper.} In Section~\ref{sec:main}
we state the assumptions on the problem data and resume previous results. 
In Section~\ref{sec:CP_ept} we study the optimization problem 
$(CP)_\ept$ as $\tau$ and $\eps$ are strictly positive and prove necessary optimality conditions, whereas in Section~\ref{sec:asymp} we solve the optimization problems 
$\CPeps$, $\CPtau$, and $(\ov{CP})$ through asymptotic arguments.

\section{Assumptions and previous results}
\label{sec:main}
\subsection{Assumptions}
Throughout the paper, $\Omega\subset\erre^d$, ($d=2,3$)
is a smooth bounded domain with boundary $\Gamma$, and $T>0$ is a fixed final time.
For every $t\in[0,T]$, we use the classical notation 
\begin{align*}
	Q_t:=\Omega\times(0,t)\,,
	\quad 
	\Sigma_t:=\Gamma\times(0,t)\,,
	\quad 
	Q_t^T:=\Omega\times(t,T)\,,
\end{align*}
and put for convenience
\begin{align*}
	Q:=Q_T\,,
	\quad 
	\Sigma:=\Sigma_T\,.
\end{align*}
For any Banach space $X$, we denote its dual space by $X^*$, 
the associated duality pairing by $\<{\cdot},{\cdot}>$ and if $X$ is a Hilbert space, 
we denote its inner product by $(\cdot, \cdot)_X$. 
For $1 \leq p \leq \infty$ and $k \geq 0$ we denote the usual Lebesgue and 
Sobolev spaces on $\Omega$ by $L^p(\Omega)$ and $W^{k,p}(\Omega)$, 
along with the corresponding norms $\norma{\cdot}_{p}$ and $\norma{\cdot}_{W^{k,p}(\Omega)}$.
For the case $p = 2$, these spaces become Hilbert spaces and 
we use the notation $H^k(\Omega) = W^{k,2}(\Omega)$.
Then, we define the functional spaces
\[
  H:=L^2(\Omega)\,, \qquad V:=H^1(\Omega)\,, \qquad
  W:=\{v\in H^2(\Omega): \partial_{\bf n}v=0\;\text{a.e.~on } \partial\Omega\}\,,
\]
endowed with their natural norms $\norma{\cdot}:=\norma{\cdot}_2$, $\norma{\cdot}_V$, and $\norma{\cdot}_W$.
The space $H$ will be identified to its dual, as usual, so that 
we have the following continuous, dense, and compact embeddings:
\[
  W\hookrightarrow V\hookrightarrow H\hookrightarrow V^*\,.
\]

Our starting point to solve the above-mentioned optimal control problems
are the well-posedness results of \statesys\ obtained in \cite{SS},
as well as the related addressed asymptotic analyses. 
In this direction, we postulate the following structural assumptions.

\begin{enumerate}[label={\bf A\arabic{*}}, ref={\bf A\arabic{*}}]
\item \label{ass:1}
	$\A,\B$ and $\P_{\rm max}, \chi_{\rm max}, \eta_{\rm max}, \CC_{\rm max}$ 
	are nonnegative constants.
\item \label{ass:2}
	$f:\erre\to\mathopen[0,+\infty\mathclose)$ is bounded and Lipschitz continuous.
\item \label{ass:3}
	 $\sigma_S\in L^\infty(Q)\,\,$ with $\,\,0\leq \sigma_S(\x,t) 
	 \leq 1 \quad\text{for a.e.~}(\x,t)\in Q$.
\item \label{ass:4}
	$F:(-\ell,\ell)\to[0,+\infty)$ is of class $C^4$, where $\ell\in(0,+\infty]$, and 
	\[
	F'(0)=0\,, \qquad
	\lim_{r\to(\pm \ell)^\mp}\left[F'(r) - \chi_{\rm \max}\eta_{\rm max} r\right]=\pm \infty\,. 
	\]
	Note that the latter condition allows both for 
	the logarithmic potential and 
	for any polynomial super-quadratic potential. Nonetheless,
	potentials of double-obstacle type are excluded.
	Let us note also that the limiting condition at $\pm\ell$
	is satisfied in particular by any $\eta\in[0,\eta_{\rm max}]$ and $\chi\in[0,\chi_{\rm max}]$.
\item \label{ass:5}
	The kernel $J\in W^{1,1}_{loc}(\erre^d)$ is even, i.e.~$J(\x)=J(-\x)$ 
	for almost every $\x\in\erre^d$.
	For any measurable $v:\Omega\to\erre$ we employ the notation
  	\[
 	(J*v)(\x):=\int_\Omega J(\x-\y)v(\y)\, \dy\,, \quad \x\in\Omega\,,
 	\]
	and set $a:=J*1$. Moreover, we assume that 
	\[
  	a_*:=\inf_{\x\in\Omega}\int_\Omega J(\x-\y)\,\dy=\inf_{\x\in\Omega}a(\x)\geq0\,,
 	 \]
 	 \[
  	a^*:= \sup_{\x\in\Omega} \int_\Omega |J(\x-\y)|\,\dy < +\infty, \qquad
 	 b^*:= \sup_{\x\in\Omega} \int_\Omega |\nabla J(\x-\y)|\, \dy < +\infty\,.
 	 \]
  	We set $c_a:=\max\{a^*-a_*,1\}$ and
  	we suppose that 
 	 there is a constant $C_0>0$ such that 
 	 \[
 	 a_* + F''(r)\geq C_0 \qquad\forall\,r\in (-\ell,\ell)\,.
 	 \]
	 Let us stress that this condition allows for non-convex potentials $F$,
	 as $a_*$ might be strictly greater than $C_0$.
\item \label{ass:6}
	The positive constants $\eps$ and $\tau$ are such that 
	$\eps\in(0,\eps_0)$ and $\tau\in(0,\tau_0)$, where
	$\eps_0$ and $\tau_0$ are defined as
	\[
	\eps_0:=\min\left\{\frac1{4c_a}, \frac{1}{\max\{1,a^*-\min\{a^*, C_0\}\}},
  	\frac{2C_0}{3(a^*+b^*)^2}\right\}\,, \qquad \tau_0:=1\,.
	\]
	This is only a technical requirement on the ``smallness'' of the perturbation
	parameters.
\item \label{ass:7}
	The convolution kernel satisfies the additional regularity property:
	\[
	\text{either}\qquad J\in W^{2,1}(\mathcal B_R) \qquad\text{or}\qquad
	\text{$J$ is admissible in the sense of Definition~\ref{def:adm}}\,,
	\]
	where $\mathcal B_R$ is the open ball in $\erre^d$ of radius 
	$R:=\operatorname{diam}(\Omega)$
	centred at $0$, and
	\begin{definition}
	\label{def:adm}
	A convolution kernel $J\in W^{1,1}_{\rm loc}(\erre^d)$ is said to be {\em admissible} if
	it satisfies:
	\begin{itemize}
	\item[(i)] $J\in C^3(\erre^d \setminus \{0\})$.
	\item[(ii)] $J$ is radially symmetric, i.e.~$J(\cdot)=\widetilde{J}(|\cdot|)$
	for a non-increasing $\widetilde J:\erre_+\to\erre$.
	\item[(iii)] there exists $R_0>0$ such that
	$r\mapsto{\widetilde J}''(r)$ and $r\mapsto{\widetilde J}'(r)/r$ are monotone on $(0,R_0)$.
	\item[(iv)] there exists $C_d>0$ such that 
	$|D^3 J(\x)|\leq C_d|\x|^{d-1}$ for every $\x\in\erre^d\setminus\{0\}$.
	\end{itemize}
	\end{definition}
	Let us recall that the $W^{2,1}$ regularity condition above, 
	despite being probably the most natural,
	prevents some relevant cases of kernels 
	such as the Newtonian or the Bessel potential from being considered.
	By contrast, these can be included 
	by using the notion of admissible kernel:
	see \cite{FG1,gal-gior-grass} and \cite[Def.~1]{bed-rod-bert} for details.
\end{enumerate}

When dealing with the aforementioned optimal control problems, 
we postulate that the cost 
functional $\J$ and the space of admissible controls $\Uad$ are defined by \eqref{cost}, and 
\eqref{Uad} and that the following are fulfilled.
\begin{enumerate}[label={\bf C\arabic{*}}, ref={\bf C\arabic{*}}]
  \item \label{ass:op:1}
  The target functions $\phO:\Omega \to \erre$ and $\phQ:Q \to \erre$ verify 
  $\phO \in \Lx2$ and
  $\phQ \in L^2(Q)$.
  \item \label{ass:op:2}
  $\bO,\bQ,\aP,\achi,\aeta,\aC$ are nonnegative constants not all zero.
  \item \label{ass:op:3}
  $\P_*,\chi_*,\eta_*,\CC_*$ are nonnegative constants.
  \item \label{ass:op:4}
  $f \in C^2_b(\erre)$, i.e.~$f$ is twice differentiable and $f$, $f'$ and $f''$ are bounded.
\end{enumerate}

\subsection{Well-posedness of the state system}
The assumptions \ref{ass:1}--\ref{ass:7}
above allow to prove the strong well-posedness of system \statesys\,
and to discuss the asymptotic behavior of the system as $\eps$ and $\tau$ go to zero,
as we will recall below.
Let us mention that in \cite{SS} 
weak well-posedness of the state system \statesys\ is also addressed, 
under less stringent assumptions on the data.
However, since in this framework we are interested in solving 
the optimal control problem $(CP)_\ept$,
we are forced to work with strong solutions instead,
which possess better stability properties with respect to the involved parameters. 
In particular, unlike weak solutions, strong solutions allow to consider 
also positive chemotaxis and active transport coefficients:
for further details, we refer to \cite{SS}.

Let us recall here the
the well-posedness of the state system \statesys\ in the strong sense.

\begin{theorem}[Strong well-posedness of the state system: $\ept >0$]
  \label{THM:WP:STRONG}
  Assume conditions \ref{ass:1}--\ref{ass:7}, let
  $\eps\in(0,\eps_0)$ and $\tau\in(0,\tau_0)$, and suppose that the
  initial data $(\varphi_0, \mu_0, \sigma_0)$ verify
  \begin{equation}
    \label{ass:in:data}
  \varphi_0\in H^2(\Omega)\,,\quad
  \mu_0,\sigma_0\in V \cap L^\infty(\Omega)\,,\quad
  \exists\,r_0\in(0,\ell):\;
  \norma{\varphi_0}_{L^\infty(\Omega)}\leq r_0\,.
  \end{equation}
  Then, for every $(\P,\chi,\eta,\CC)\in \Uad$
  there exists a unique strong solution $(\varphi,\mu,\sigma)$ such that
  \begin{align}
    \label{phi_reg2}
    &\varphi\in W^{1,\infty}(0,T; V)\cap H^1(0,T; H^2(\Omega))\,, \quad\partial_t\varphi\in L^\infty(Q)\,,
    \quad \eta \varphi\in L^2(0,T; W)\,,\\
    \label{phi_sep}
    &\exists\,r^*\in(r_0,\ell):\quad\sup_{t\in[0,T]}\norma{\varphi(t)}_{L^\infty(\Omega)}\leq r^*\,,\\
    \label{mu_reg2}
    &\mu, \sigma \in H^1(0,T; H)\cap L^\infty(0,T; V)\cap L^2(0,T; W)\cap L^\infty(Q)\,,
  \end{align}
  which fulfils equations \eqref{state:1}--\eqref{state:5} pointwise in $Q$.
Furthermore, there exists a constant $K>0$, 
depending only on the structural data in {\bf A1--A7}, 
on the initial data, and on $\eps$ and $\tau$, such that
\begin{align}
	& \non
	 \norma{\dt\ph}_{L^\infty(Q)}
	+ \norma{\ph}_{W^{1,\infty}(0,T; V)\cap H^1(0,T; H^2(\Omega))}
	+ \norma{\eta \ph}_{\L2 W}
	\\ & \quad \non
	+ \norma{\mu}_{H^1(0,T; H)\cap L^\infty(0,T; V)\cap L^2(0,T; W)\cap L^\infty(Q)}
	\\ & \quad
	+ \norma{\s}_{H^1(0,T; H)\cap L^\infty(0,T; V)\cap L^2(0,T; W)\cap L^\infty(Q)}
	\leq K \qquad\forall\,(\P,\chi,\eta,\CC)\in\Uad\,.
	\label{estimate:strong}
\end{align}
Moreover, for any pair of initial data 
  $\{(\varphi_0^i, \mu_0^i, \sigma_0^i)\}_{i=1,2}$
  satisfying \eqref{ass:in:data}
  there exists a constant $K>0$,
  depending on the structural data, $\eps$, and $\tau$,
  such that, for every pair of parameters 
  $\{(\P^i, \chi^i, \eta^i, \CC^i)\}_{i=1,2}\in\Uad$,
  and for any respective strong solutions
  $\{(\varphi_i, \mu_i, \sigma_i)\}_{i=1,2}$
  satisfying \eqref{phi_reg2}--\eqref{mu_reg2}, it holds that
  \begin{align*}
  &
  \norma{\varphi_1-\varphi_2}_{W^{1,\infty}(0,T; V)\cap H^1(0,T; H^2(\Omega))}
  + \norma{\mu_1-\mu_2}_{H^1(0,T; H)\cap L^\infty(0,T; V)\cap L^2(0,T; W)}\\
  &\qquad
  +\norma{\sigma_1-\sigma_2}_{H^1(0,T; H)\cap L^\infty(0,T; V)\cap L^2(0,T; W)}\\
  &\quad\leq K\left( \norma{\varphi_0^1-\varphi_0^2}_{H^2(\Omega)}
  +\norma{\mu_0^1-\mu_0^2}_V
  +\norma{\sigma_0^1-\sigma_0^2}_V \right)
  \\ & \qquad
   + K\left(|\P^1-\P^2 |+ |\chi^1-\chi^2 |+ |\eta^1-\eta^2 |+ |\CC^1-\CC^2 | \right). 
  \end{align*}
\end{theorem}

The (strong) well-posedness presented in the theorem above allows us to define the so-called
{\it control-to-state} operator $\S$,
assigning to every given admissible control $(\P,\chi,\eta,\CC)$ 
the unique corresponding state. Namely, we have
\begin{align*}
	\S: (\P,\chi,\eta,\CC) \mapsto (\ph,\m,\s),
\end{align*}
where $(\ph,\m,\s)$ is the unique solution to \statesys\ obtained from Theorem~\ref{THM:WP:STRONG}.
Moreover, let us draw a \sfw\ consequence of the separation result \eqref{phi_sep} established above
which will be useful later on.
\begin{corollary}
\label{COR:SEP}
In addition to the assumptions of Theorem \ref{THM:WP:STRONG}, suppose that
$F \in C^k(-\ell,\ell)$ for some $k\geq 4$.
Then, there exists a positive constant $K$, 
depending only on the structural data in {\bf A1--A7}, 
on the initial data, and on $\eps$ and $\tau$,
such that
\begin{align}
\label{separation:est}
\norma{\ph}_{L^\infty(Q)} + \max_{i=0,\ldots,k} \norma{F^{(i)}(\ph)}_{L^\infty(Q)} \leq K
\qquad\forall\,(\P,\chi,\eta,\CC)\in\Uad\,.
\end{align}
\end{corollary}

\subsection{Asymptotic behavior of the state system}
Let us recollect here the known results established in \cite{SS} 
concerning the asymptotic behavior of the state system \statesys\ as $\ept \to 0$. 
Despite part of the following results work also under more general 
assumptions on the potential (see \cite{SS}), 
in some cases it will be useful to introduce also the following assumption 
on the potential $F$.
\begin{enumerate}[label={\bf A\arabic{*}}, ref={\bf A\arabic{*}}, start=8]
\item \label{ass:8}
There exist two positive constants $c_F$ and $C_F$ such that
\begin{align*}
|F''(r)| \leq C_F ( 1+ |r|^2)\,,
\qquad F(r)\geq c_F|r|^4 - C_F\qquad\forall\,r \in \erre\,.
\end{align*}
It is worth noting that these conditions are met by the classical regular potential \eqref{Fpol}, whereas prevent the singular choices \eqref{Flog} and \eqref{Fdobs} to be considered.
\end{enumerate}

\begin{theorem}[Asymptotics: $\eps\to0$]
\label{THM:PREV:ASY:epszero}
  Assume \ref{ass:1}--\ref{ass:8}, let $\tau\in(0,\tau_0)$,
  and suppose that
  \begin{align}
  \label{coeff_eps0}
  &(\P_\tau,\chi_\tau,\eta_\tau,\CC_\tau)\in\Uad\,, \qquad \eta_\tau=0\,,\\
  \label{init_eps0}
  &\varphi_{0,\tau}\in V\,, \qquad F(\varphi_{0,\tau})\in L^1(\Omega)\,,\qquad
   \sigma_{0,\tau}\in H\,, 
   \qquad 
  0 \leq \s_{0,\tau}(\x) \leq 1 \ \ \hbox{for $a.e. \,\, \x \in \Omega$}\,.
  \end{align}
  For $\eps\in(0, \eps_0)$, let 
  the coefficients $(\P_\ept,\chi_\ept,\eta_\ept,\CC_\ept)$ and
  the initial data $(\varphi_{0,\ept}, \mu_{0,\ept}, \sigma_{0,\ept})$
  satisfy \eqref{coeff_eps0}--\eqref{init_eps0} and \eqref{ass:in:data}.
  Suppose also that, as $\eps\to0$, it holds
  \begin{align}
  \non
  &(\P_\ept,\chi_\ept,\eta_\ept,\CC_\ept)\to(\P_\tau,\chi_\tau,\eta_\tau,\CC_\tau),\\
  \label{conv_init_eps0}
  &\varphi_{0,\ept}\to\varphi_{0,\tau} \quad  \text{weakly in } V, \qquad
  \sigma_{0,\ept}\to\sigma_{0,\tau} \quad \text{strongly in } H, 
  \end{align}
  and that there exists a constant $M_0>0$, independent of $\eps$, such that,
  for all $\eps\in(0,\eps_0)$,
  \begin{equation}
  \label{bound_init_eps0}
  \eps^{1/2}\norma{\mu_{0,\ept}} 
  + \norma{F(\varphi_{0,\ept})}_{L^1(\Omega)}
   + \eps^{1/4}\left(\norma{\mu_{0,\ept}}_V + 
    \norma{\sigma_{0,\ept}}_V
    + \norma{F'(\varphi_{0,\ept})}\right)
    \leq M_0.
  \end{equation}
  Denote by $(\ph_\ept, \mu_\ept, \sigma_\ept)$ 
  the respective unique strong solution to the system \statesys\
  obtained from Theorem~\ref{THM:WP:STRONG},
  with respect to the 
  the coefficients $(\P_\ept,\chi_\ept,\eta_\ept,\CC_\ept)$ and the initial
  data $(\varphi_{0,\ept}, \mu_{0,\ept}, \sigma_{0,\ept})$.
  Then, there exists a unique triplet 
  $(\ph_\tau,\mu_\tau, \sigma_\tau)$, with
  \begin{align*}
	&\ph_\tau \in \H1 H \cap \L\infty V\,,
	\qquad
  	\mu_\tau \in \L2 {V}\,, \\
  	&\sigma_\tau \in \H1 {V^*} \cap \L2 V\cap L^\infty(Q)\,, 
  	\\
	& 0\leq\sigma_\tau(\x,t)\leq1 \quad \text{for a.e.~}\x\in\Omega\,,\quad\forall\,t\in[0,T]\,,\end{align*}
  such that 
  \begin{align*}
  & \<\partial_t \varphi_\tau, v>
   + \iO\nabla  \mu_\tau \cdot \nabla v
  = \iO(\P_\tau\sigma_\tau - \A)f(\varphi_\tau)v\,, \\
  & \<\partial_t\sigma_\tau,w>
  + \iO \nabla\sigma_\tau \cdot \nabla w
  + \B \iO (\sigma_\tau-\sigma_S)w
  + \CC_\tau\iO \sigma_\tau f(\varphi_\tau)w
  = 0\,,
\end{align*}
for every $v,w\in V$, almost everywhere in $(0,T)$, and
\begin{align*}
    &\mu_\tau=\tau\partial_t\varphi_\tau
    + a\varphi_\tau - J*\varphi_\tau 
    + F'(\varphi_\tau) 
    - \chi_\tau\sigma_\tau &&\text{a.e.~in } Q\,,
    \\
    &\varphi_\tau(0)=\varphi_{0,\tau}\,, \quad
    \sigma_\tau(0)=\sigma_{0,\tau} &&\text{a.e.~in } \Omega\,.
\end{align*}
  Moreover, as $\eps\to0$ it holds that 
\begin{align}
    \label{epstozero:1}
  	\ph_\ept  &\to \ph_\tau &&\hbox{weakly-$^*$ in $\H1 H \cap \L\infty V$,}\\
  	\label{epstozero:2}
	\mu_\ept  &\to \mu_\tau &&\hbox{strongly in  $\L2 V$,}\\
  	\label{epstozero:3}
	\sigma_\ept  &\to \sigma_\tau 
	\quad &&\hbox{weakly-$^*$ in $\H1 {V^*}  \cap \L2 V \cap L^\infty(Q)$,}
	\\
  	\label{epstozero:5}
  \eps \mu_\ept  &\to 0  &&\hbox{strongly in $C^0([0,T]; H)\cap \L2 V$,}
\end{align}
\Accorpa\epstozero {epstozero:1} {epstozero:5}
and 
\begin{align}
	\label{epstozero:strong:1}
	\ph_\ept &\to\ph_\tau &&\text{strongly in } 
	C^0([0,T]; H)\,, \\
  	\label{epstozero:strong:2}
  \sigma_\ept& \to\sigma_\tau &&\text{strongly in } L^\infty(0,T; H)\cap L^2(0,T; V)\,.
\end{align}  
\end{theorem}

In a similar fashion we have the following result.
\begin{theorem}[Asymptotics: $\tau\to0$]
\label{THM:PREV:ASY:tauzero}   
Assume \ref{ass:1}--\ref{ass:7}, let $\eps\in(0,\eps_0)$, and suppose that
   \begin{align}
   \label{coeff_tau0}
   &(\P_\eps,\chi_\eps,\eta_\eps,\CC_\eps)\in\Uad\,, \qquad
   0\leq \chi_\eps < \sqrt{c_a}\,, \qquad
   (\chi_\eps+\eta_\eps+4c_a\chi_\eps)^2<8c_aC_0+4\chi_\eps\eta_\eps\,,\\
   \label{init_tau0}
  &\varphi_{0,\eps},\mu_{0,\eps},\sigma_{0,\eps} \in H\,,\qquad F(\varphi_{0,\eps})\in L^1(\Omega)\,.
  \end{align}
  For all $\tau\in(0,\tau_0)$, let the coefficients
  $(\P_\ept,\chi_\ept, \eta_\ept, \CC_\ept)$ and
  the initial data $(\varphi_{0,\eps\tau}, \mu_{0,\eps\tau}, \sigma_{0,\eps\tau})$
  satisfy \eqref{coeff_tau0}--\eqref{init_tau0} and \eqref{ass:in:data}.
 Suppose also that, as $\tau\to0$, it holds
  \begin{align}
  \non
  &(\P_\ept, \chi_\ept, \eta_\ept, \CC_\ept)\to (\P_\eps,\chi_\eps,\eta_\eps,\CC_\eps),\\
  \label{conv_init_tau0}
  &\varphi_{0,\ept}\to\varphi_{0,\eps}, \qquad
  \mu_{0,\ept}\to\mu_{0,\eps}, \qquad
  \sigma_{0,\ept}\to\sigma_{0,\eps} \quad\text{strongly in } H,
  \end{align}
  and that there exists a constant $M_0>0$, independent of $\tau$, such that,
  for all $\tau\in(0,\tau_0)$,
  \begin{equation}\label{bound_init_tau0}
  \tau^{1/2}\norma{\varphi_{0,\ept}}_V 
  + \norma{F(\varphi_{0,\ept})}_{L^1(\Omega)}\leq M_0.
  \end{equation}
  Denote by $(\ph_\ept,\mu_\ept, \sigma_\ept)$ 
  the respective unique strong solution to the system
  \statesys\ obtained from Theorem \ref{THM:WP:STRONG}
  with respect to the coefficients $(\P_\ept,\chi_\ept, \eta_\ept, \CC_\ept)$ and
  the initial data $(\varphi_{0,\ept}, \mu_{0,\ept}, \sigma_{0,\ept})$.
  Then, there exists a triplet
  $(\ph_\eps,\mu_\eps, \sigma_\eps)$, with
  \begin{align*}
	&\ph_\eps,\mu_\eps \in L^\infty(0,T; H) \cap L^2(0,T; V)\,,
	\qquad
  	\eps\mu_\eps + \varphi_\eps \in H^1(0,T; V^*)\cap \L2 V\,, \\
  	&\sigma_\eps \in \H1 {V^*} \cap \L2 V\,,
\end{align*}
  such that 
  \begin{align*}
  & \<\partial_t(\eps\mu_\eps+\varphi_\eps), v>
   + \iO\nabla  \mu_\eps \cdot \nabla v
  = \iO(\P_\eps\sigma_\eps - \A)f(\varphi_\eps)v\,, \\
  & \<\partial_t\sigma_\eps,w>
  + \iO \nabla \sigma_\eps \cdot \nabla w
  + \B \iO (\sigma_\eps-\sigma_S) w
  + \CC_\eps\iO \sigma_\eps f(\varphi_\eps)w
  = \eta_\eps\int_\Omega\nabla\varphi_\eps\cdot\nabla w\,,
\end{align*}
for every $v,w\in V$, almost everywhere in $(0,T)$, and
\begin{align*}
    \mu_\eps=a\varphi_\eps - J*\varphi_\eps 
    + F'(\varphi_\eps)
    - \chi_\eps\sigma_\eps \qquad&\text{a.e.~in } Q\,, \\
    \varphi_\eps(0)=\varphi_{0,\eps}\,, \quad
    \sigma_\eps(0)=\sigma_{0,\eps} \qquad&\text{a.e.~in } \Omega\,.
\end{align*}
  Moreover, as $\tau\to0$, along a non-relabelled subsequence it holds that 
  \begin{align}
  	\label{tautozero:1}
  	\ph_\ept  &\to\ph_\eps&&\hbox{weakly-$^*$ in } L^\infty(0,T; H) \cap \L2 V\,,\\
  	\label{tautozero:2}
  	\mu_\ept  &\to\mu_\eps &&\hbox{weakly-$^*$ in } L^\infty(0,T; H) \cap \L2 V\,,\\
	\label{tautozero:3}
	\eps\mu_\ept+\varphi_{\ept}&\to\eps\mu_\eps+\varphi_\eps
	&&\text{weakly in } H^1(0,T; V^*)\cap L^2(0,T; V)\,,\\
  	\label{tautozero:4}
  	\sigma_\ept  &\to \sigma_\eps &&\hbox{weakly in } \H1 {V^*}  \cap \L2 V\,,\\
  	\label{tautozero:6}
  	\tau\varphi_\ept&\to 0 &&\text{strongly in } H^1(0,T; H)\cap L^\infty(0,T; V)\,,
  \end{align}
  \Accorpa\tautozero {tautozero:1} {tautozero:6}
 and
  \begin{align}
  \label{tautozero:strong:1}
  \ph_\ept&\to\ph_\eps \quad\text{strongly in } L^2(0,T; H)\,, \qquad
  \mu_\ept\to\mu_\eps\quad\text{strongly in } L^2(0,T; H)\,,\\
  \label{tautozero:strong:2}
  \sigma_\ept&\to\sigma_\eps \quad\text{strongly in } C^0([0,T]; V^*)\cap L^2(0,T; H)\,.
  \end{align}
  Furthermore, if also  
  \begin{align*}
  \eta_\ept=0\,, \qquad
  0 \leq \s_{0,\ept}(\x) \leq 1 \ \ \hbox{for $a.e. \, \, \x \in \Omega$} \qquad\forall\,\tau\in(0,\tau_0)\,,
  \end{align*}
  then the solution $(\ph_\eps,\mu_\eps, \sigma_\eps)$
  of the system \statesys\ with $\tau=0$ is unique, 
  all the convergences hold along the entire sequence $\tau\to0$, and
  \begin{align*}
  0 \leq \sigma_\eps (\x,t) \leq 1 \quad
  \hbox{for $a.e. \, \,  \x \in \Omega, \quad\forall\, t \in [0,T]$}\,.
\end{align*}
\end{theorem}

Lastly, the joint asymptotics is given in the following statement.
\begin{theorem}[Asymptotics: $\ept\to 0$]
\label{THM:PREV:ASY:epstautozero}   
Assume \ref{ass:1}--\ref{ass:8} and suppose that
\begin{align}
  \label{coeff_ept0}
  &(\P,\chi,\eta, \CC)\in\Uad\,, \qquad \eta=0\,,\qquad
   0\leq \chi < \sqrt{c_a}\,, \qquad
   (\chi+\eta+4c_a\chi)^2<8c_aC_0+4\chi\eta\,,\\
  \label{init_ept0}
  &\varphi_0, \sigma_0 \in H\,, \qquad F(\varphi_0)\in L^1(\Omega)\,.
\end{align}
 For every $\eps\in(0,\eps_0)$ and $\tau\in(0,\tau_0)$, let
 the coefficients $(\P_\ept, \chi_\ept, \chi_\ept, \CC_\ept)$
 and the initial data $(\varphi_{0,\ept}, \mu_{0,\ept}, \sigma_{0,\ept})$
 satisfy \eqref{coeff_ept0}--\eqref{init_ept0} and \eqref{ass:in:data}.
 Suppose also that, as $(\eps,\tau)\to(0,0)$, it holds
  \begin{align}
  \non
  &(\P_\ept, \chi_\ept, \chi_\ept, \CC_\ept)\to(\P,\chi,\eta, \CC),\\
  &\label{conv_init_ept0}
  \varphi_{0,\ept}\to\varphi_0 \quad\text{strongly in } H\,, \qquad
  \sigma_{0,\ept}\to\sigma_0 \quad\text{strongly in } H\,,
  \end{align}
  and that there exists a constant $M_0>0$, independent of $(\eps,\tau)$, such that
  \begin{align}
  \label{bound_init_ept0}
  &\tau^{1/2}\norma{\varphi_{0,\ept}}_V +
  \eps^{1/2}\norma{\mu_{0,\ept}}
  + \norma{F(\varphi_{0,\ept})}_{L^1(\Omega)}\leq M_0
  \qquad\forall\,(\eps,\tau)\in(0,\eps_0)\times(0,\tau_0).
  \end{align}
  Denote by $(\ph_\ept,\mu_\ept, \sigma_\ept)$ 
  the respective unique strong solution to the system
  \statesys\ obtained from Theorem \ref{THM:WP:STRONG},
  with respect to the coefficients $(\P_\ept, \chi_\ept, \eta_\ept, \CC_\ept)$
  and the initial data $(\varphi_{0,\eps\tau},\mu_{0,\eps\tau},\sigma_{0,\eps\tau})$.
  Then, there exists a unique triplet $(\varphi,\mu,\sigma)$, with 
  \begin{align*}
	&\ph \in H^1(0,T; V^*) \cap L^2(0,T; V)\,,\qquad
         \mu \in L^2(0,T; V)\,, \\
  	&\sigma \in \H1 {V^*} \cap \L2 V\cap L^\infty(Q)\,, \qquad
	0\leq\sigma(\x,t)\leq1 \quad\text{for a.e.~}\x\in\Omega\,,\;\forall\,t\in[0,T]\,,
\end{align*}
  such that
\begin{align*}
  & \<\partial_t\varphi, v>
   + \iO\nabla  \mu \cdot \nabla v
  = \iO(\P\sigma - \A)f(\varphi)v\,, \\
  & \<\partial_t\sigma,w>
  + \iO \nabla \sigma\cdot \nabla w
  + \B \iO (\sigma-\sigma_S)w
  + \CC\iO \sigma f(\varphi)w
  = 0
\end{align*}
  for every $v,w\in V$, almost everywhere in $(0,T)$, and
  \begin{align*}
  \mu =a\varphi - J*\varphi +F'(\varphi) - \chi\sigma 
  \qquad&\text{a.e.~in } Q\,, \\
  \varphi(0)=\varphi_{0}\,, \quad
    \sigma(0)=\sigma_{0} \qquad&\text{a.e.~in } \Omega\,.
  \end{align*}
  Moreover, as $(\eps,\tau)\to(0,0)$ it holds that 
  \begin{align}
  	\label{epstautozero:1}
  	\ph_\ept  &\to \ph &&\hbox{weakly-$^*$ in } L^\infty(0,T; H) \cap \L2 V\,,\\
  	\label{epstautozero:2}
	\mu_\ept  &\to \mu &&\hbox{weakly in } \L2 V\,,\\
	\label{epstautozero:3}
	\eps\mu_\ept+\varphi_{\ept} &\to \varphi
	&&\text{weakly in } H^1(0,T; V^*)\cap L^2(0,T; V)\,,\\
  	\label{epstautozero:4}
  	\sigma_\ept  &\to \sigma &&\hbox{weakly-$^*$ in } \H1 {V^*}  \cap \L2 V \cap L^\infty(Q)\,,\\
	\eps^{1/2}\mu_\ept &\to 0 &&\text{weakly-$^*$ in } L^\infty(0,T; H)\,,\\
	\label{epstautozero:5}
	\eps\mu_\ept &\to 0 &&\text{strongly in } C^0([0,T]; H) \cap \L2 V\,,\\
  	\label{epstautozero:6}
  	\tau\varphi_\ept &\to 0&&\text{strongly in } H^1(0,T; H)\cap L^\infty(0,T; V)\,,
  \end{align}
  \Accorpa\epstautozero {epstautozero:1} {epstautozero:6}
  and
  \begin{align}
  \label{epstautozero:strong:1}
  \ph_\ept&\to \ph &&\text{strongly in } L^2(0,T; H)\,, 
  \\ 
  \sigma_\ept&\to \sigma && \text{strongly in } C^0([0,T]; V^*)\cap L^2(0,T; H)\,.
   \label{epstautozero:strong:2}
  \end{align}
\end{theorem}

\section{The optimization problem $(CP)_\ept$}
\label{sec:CP_ept}
This section is focused on the analysis of the optimal control problem 
$(CP)_\ept$, when $\ept>0$ are given.
More specifically, we show existence of optimal controls and 
necessary first-order condition for optimality.
Throughout the whole Section~\ref{sec:CP_ept}, 
we work with $\eps\in(0,\eps_0)$ and $\tau\in(0,\tau_0)$ fixed.
For this reason, we omit for brevity notation for the dependence on $\ept$
for the variables in play.

\subsection{Existence of a minimizer}
The first problem that 
we address concerns the existence of a minimizer of the optimal control problem $(CP)_\ept$,
with $\eps,\tau>0$ being fixed.
Its proof is rather standard and follows as a consequence of the direct method of calculus of variations.

\begin{theorem}[Existence of a minimizer]
\label{THM:EXISTECE:MIN}
Suppose that \ref{ass:1}--\ref{ass:7} and \ref{ass:op:1}--\ref{ass:op:3} are fulfilled.
Then, the optimization problem $(CP)_\ept$ admits a minimizer.
\end{theorem}
\begin{proof}
Without loss of generality we assume that all the constants $\aP,\achi,\aeta,\aC$ are positive.
In fact, if this is not the case, we can consider the corresponding control $\P,\chi,\eta$ and/or $\CC$ as a prescribed constant, redefine $\Uad$ accordingly and argue in the same way.
To begin with, notice that the cost functional is nonnegative so that we can 
consider a minimising sequence of elements of $\Uad$.
Namely, we take the minimising sequence $\{(\P_n,\chi_n,\eta_n,\CC_n)\}_n \subset \Uad$ and the corresponding 
sequence of states $\{(\ph_n,\m_n,\s_n)\}_n:=\{\S(\P_n,\chi_n,\eta_n,\CC_n)\}_n$
 all related to the same initial data 
$(\ph_0,\mu_0,\s_0)$.
Namely, we have
\begin{align*}
	0 \leq \ov \lam &:=\inf \Big \{\J(\ph,\P,\chi,\eta,\CC) \big|\; (\P,\chi,\eta,\CC)\in \Uad, \, \ph=\S_1(\P,\chi,\eta,\CC) \Big \}\,,
\end{align*}
and
\[
  \J(\ph_n,\P_n,\chi_n,\eta_n,\CC_n)\searrow \ov\lam\,,
\]
where $\S_1$ denotes the first component of the solution operator $\S$.
Noting that the bounds \eqref{estimate:strong}
are uniform in $n$ {thanks to the structure of $\Uad$,}
invoking standard compactness results (cf., e.g., \cite[Sec.~8, Cor.~4]{Simon})
it is \sfw\ to obtain
the existence of limits $\bph,$ $\opt\in \Uad$, and a not relabelled subsequence such that, as $n\to\infty$,
\begin{align*}
	\ph_n & \to \bph \quad \hbox{weakly-$^*$ in $\W{1,\infty} V \cap \H1{\Hx2},$}
	\\ & \hspace{2cm}
	\hbox{and strongly in $\C0 {C^0(\ov{\Omega})}$},
	\\
	\m_n & \to \bm \quad \hbox{weakly-$^*$ in $\H1 H \cap \L \infty V 
	\cap \L2 W \cap L^\infty(Q),$}
	\\
	\s_n & \to \bs \quad \hbox{weakly-$^*$ in $\H1 H \cap \L \infty V 
	\cap \L2 W \cap L^\infty(Q),$}
	\\
	\P_n & \to \ov \P, \quad \chi_n \to \ov \chi, \quad 
	\eta_n \to \ov \eta, \quad \CC_n \to \ov \CC.
\end{align*}
It is then a standard matter to pass to the limit in the variational formulation of \statesys\ written for 
$\{(\ph_n,\m_n,\s_n)\}_n$ to deduce that $(\bph,\bm,\bs)= \S \opt$.
Lastly, the weak sequential lower semincontinuity of $\J$ entails that $\opt$ is 
a minimizer for $(CP)_\ept$ together with its corresponding state $(\bph,\bm,\bs)$.
\end{proof}

\subsection{Linearized system}
We study here the linearized system, which can be formally obtained 
by differentiating the state system \eqref{state:1}--\eqref{state:5}
with respect to the control in a certain direction.

First, we fix some preliminary notation: let $(\P,\chi,\eta,\CC)\in\Uad$ be fixed, set 
\[
  (\bph, \bm,\bs):=\S(\P,\chi,\eta,\CC)\,,
\] 
and consider an arbitrary increment 
\[
\bh:=(h_\P,h_{\chi},h_{\eta},h_\CC)\in\erre^4 \qquad\text{such that}\qquad
(\P,\chi,\eta,\CC)+\bh\in\Uad\,.
\]
The variables of the linearized system are denoted by $(\xi,\nu,\zeta)$:
of course, they depend on the increment $\bh$, but we avoid 
keeping track of this explicitly for brevity of notation.

The linearized system reads:
\begin{align}
	\label{lin:1}
	&\dt(\eps  \nu+ \xi)- \Delta \nu
	=
	(\P \bs-\A) f'(\bph)\xi
	+\P\z f(\bph)  
	+ h_\P \bs f(\bph)  
	&&\hbox{in $\, Q$},
	\\
	\label{lin:2}
	&\nu
	=
	\tau\dt \xi
	+a \xi -J * \xi
	+ F''(\bph)\xi
	- \ch \z
	- h_{\ch} \bs
 	&&\hbox{in $\,Q$},
	\\
	&\dt \z
	- \Delta \z
	+\B\z
	+ \CC ( \z f(\bph) + \bs f'(\bph)\xi )
	+ h_\CC \bs f(\bph)
	\label{lin:3}
	=-\Delta(\eta\xi + h_\eta\bph)
	&&\hbox{in $\,Q$},
	\\
	&\dn \nu=\dn(\eta\xi + h_\eta\bph)=\dn \z =0
 	&&\hbox{on $\,\Sigma$},
	\label{lin:4}
	\\ 
	&\nu(0)=\xi(0)=\z(0)=0
  	&&\hbox{in $\,\Omega.$}
	\label{lin:5}
\end{align}
\Accorpa\linsys {lin:1} {lin:5}

\begin{theorem}[Well-posedness of the linearized system]
\label{THM:LIN}
Assume {\bf A1--A7} and {\bf C1--C4}.
Then, the linearized system \linsys\ admits
a unique solution $(\xi,\nu,\z)$ satisfying 
\begin{align*}
	\xi \in \H1 H \cap \L\infty V,\qquad			
	\nu,\z \in H^1(0,T; V^*)\cap \L2 V .
\end{align*}
\end{theorem}
\begin{proof}[Proof of Theorem \ref{THM:LIN}]
In what follows we proceed formally by pointing out some
a priori estimates. Anyway, it is a standard matter to perform the same computations
within a Galerkin scheme and then pass to the limit as the discretization parameter approach infinity
to deduce the same results at the continuous level.
Moreover, let us notice that the symbols $M$ and $\cd$ will denote generic constants
depending only on structural data and possibly on an additional positive constant $\d$
and may change from line to line.

\noindent {\bf First estimate:}
To begin with we add to both sides of \eqref{lin:2} the term $(c_a+2)\xi$.
Next, we multiply \eqref{lin:1} by $\nu$, the new \eqref{lin:2} by $-\dt \xi$, 
the gradient of  \eqref{lin:2} by $-\nabla \xi$, \eqref{lin:3} by $\z$, 
integrate over $Q_t$ and by parts. Adding the resulting equalities we obtain that
\begin{align*}
	& \frac {\eps}2 \norma{\nu(t)}^2
	+ \I2 {\nabla \nu}
	+ \tau \I2 {\dt \xi}
	+ \frac {\tau} 2 \IO2 {\nabla \xi}
	+ (\tfrac {c_a}2+1) \IO2 {\xi}
	\\ & \qquad
	+  \intQt (a + F''(\bph) )|\nabla \xi|^2
	+ \intQt (a \xi - J *\xi)\dt\xi
	+ \frac 12 \IO2 {\z}
	+ \B \I2 \z
	+ \I2 {\nabla\z}
	\\ & \quad
	=
	\underbrace{\intQt (\P \bs-\A) f'(\bph)\xi\nu
	+ \intQt \P \z f(\bph) \nu
	+ \intQt h_\P  \bs f(\bph) {\nu}
	- \intQt F''(\bph)\xi \dt\xi}_{=:I_1}
	\\ & \qquad
	+ \underbrace{\intQt \chi \z \dt\xi
	+ \intQt h_{\chi} \bs \dt\xi
	- \intQt (c_a+2)\xi \dt \xi}_{=:I_2}
	\\ & \qquad 
	 + \underbrace{\intQt \nabla \nu \cdot \nabla \xi
	 -  \intQt  (\nabla a) \xi \cdot \nabla \xi
	+ \intQt (\nabla J*\xi )\cdot \nabla \xi
	- \intQt F'''(\bph)\xi \nabla \bph \cdot \nabla \xi}_{=:I_3}
	\\ & \qquad 
	+  \underbrace{\intQt\chi \nabla \z \cdot \nabla \xi
	+  \intQt h_{\chi}\nabla \bs \cdot \nabla \xi 
	- \intQt \CC (\z  f(\bph) + \bs f'(\bph)\xi)\z
	- \intQt h_\CC\bs  f(\bph) \z}_{=:I_4}
	\\ & \qquad 
	+  \underbrace{\intQt \eta\nabla \xi \cdot \nabla \z
	-  \intQt h_{\eta}\Delta \bph \, \z}_{=:I_5}.
\end{align*}
The first two terms of the second line can be easily estimated by using \ref{ass:5},
which produces
\begin{align}
	 \intQt (a + F''(\bph) )|\nabla \xi|^2
	\geq  C_0 \I2 {\nabla \xi},
	\label{lin:est:1}
\end{align}
and similarly for the next term we have
\begin{align}
	\label{lin:est:2}
	\intQt (a \xi - J *\xi)\dt\xi
	\geq \frac {a_*-a^*}2 \IO2 {\xi}
	\geq - \frac{c_a}2  \IO2 {\xi}.
\end{align}
Moreover, the terms on the \rhs\ can be easily estimated using 
Young and \Holder\ inequalities, the regularity of the state variables $\bph, \bs$
expressed in \eqref{estimate:strong}, the boundedness of $f$ and $f'$,
and Corollary~\ref{COR:SEP}. In fact, for every $\delta>0$, we obtain that
\begin{align*}
	|I_1| 
	&\leq 
	M \intQt ( |\bs||\xi| +|\z|+|\bs| )|{\nu}|
	+ M \norma{F''(\bph)}_{L^\infty(Q)} \intQt |\xi||\dt\xi|
	\\ & \leq 
	\d \I2 {\dt\xi} 
	+\cd \intQt (|\xi|^2{+|\nu|^2}+ |\z|^2 +1),\\
	|I_2| 
	&\leq 
	M \intQt (|\z|+|\bs|+ |\xi|)|\dt\xi|
	\leq
	\d \I2 {\dt\xi} 
	+\cd \intQt (|\z|^2+|\xi|^2+ 1).
\end{align*}
Similarly, using the continuous embedding $V\hookrightarrow L^6(\Omega)$, 
H\"older's inequality,
the regularity $\bph$ and again Corollary \ref{COR:SEP} we find that
\begin{align*}
	|I_3|
	& \leq
	M \intQt (|\nabla \nu|+|\xi|)|\nabla \xi|
	+ M \norma{F'''(\bph)}_{L^\infty(Q)} \intQt |\xi||\nabla \bph||\nabla \xi|
	\\ & \leq 
	\d \I2 {\nabla \nu}
	+ \cd \intQt (|\xi|^2+ |\nabla \xi|^2)
	+M \int_0^t {\|\bph(s)\|_{H^2(\Omega)}}\|\xi(s)\|_{V}\|\nabla\xi(s)\|\,ds
	\\
	& \leq 
	\d \I2 {\nabla \nu}
	+ \cd \intQt (|\xi|^2+ |\nabla \xi|^2).
\end{align*}
By analogous computations, we have that
\begin{align*}
	|I_4|& \leq
	M \intQt (|\nabla \z|+|\nabla \bs|)|\nabla \xi|
	+ M \intQt (|\z|+|\bs||\xi|+ |\bs|)|\z|
	\\ & \leq 
	\d \intQt |\nabla \z|^2
	+\cd \intQt (| {\nabla \xi}|^2+ |\xi|^2 +|\z|^2 +1),
	\\
	|I_5|& \leq
	M \intQt |\nabla \xi||\nabla \z|
	+ M \intQt |\Delta \bph | |\z|
	\leq
	\d \I2 {\nabla \z} 
	+ \cd \intQt (|\z|^2+|\nabla \xi|^2+ 1).
\end{align*}
Upon collecting the above estimates, we infer that
\begin{align*}
& \frac {\eps}2 \norma{\nu(t)}^2
	+ (1- \d)\I2 {\nabla \nu}
	+ (\tau - 2 \d)  \I2 {\dt \xi}
	+ \IO2 {\xi}
	+ \frac {\tau} 2 \IO2 {\nabla \xi}
	\\ & \qquad
	+ C_0 \intQt |\nabla \xi|^2
	+ \frac 12 \IO2 {\z}
	{+\B\I2 \z}
	+ (1- 2\d)\I2 {\nabla \z}
	\\ & \quad \leq
	\cd \intQt (|\xi|^2+|\nabla \xi|^2 + |\z|^2+ 1).
\end{align*}
Then, we choose $\d:= \min \{ \tfrac \tau 4, \tfrac 14 \}$
so that Gronwall's lemma yields that
\begin{align*}
	\norma{\xi}_{\H1 H \cap \L\infty V}
	+ 
	\norma{\nu}_{\L\infty H \cap \L2 V}
	+ 
	\norma{\z}_{\L\infty H \cap \L2 V}
	\leq M.
\end{align*}

\noindent {\bf Second estimate:}
In light of the above estimate, a comparison argument
in \eqref{lin:1} and \eqref{lin:3} produces 
\begin{align*}
	\norma{\dt(\eps \nu + \xi)}_{\L2 \Vp}
	+
	\norma{\dt \z}_{\L2 \Vp}
	\leq M,
\end{align*}
hence also 
\[
  \|\partial_t\nu\|_{L^2(0,T; V^*)}\leq M\,.
\]

\noindent {\bf Conclusion:} It is clear that these estimates 
are enough to pass to the limit in the linearized system.
Furthermore,
it is worth noting that the uniqueness directly follows from 
the linearity of the system and the estimates above.
The proof of Theorem \ref{THM:LIN} is then concluded.
\end{proof}

\subsection{\Frechet\ differentiability of $\S$}
\begin{theorem}[\Frechet\ differentiability]
	\label{THM:FRECHET}
	Assume {\bf A1--A7} and {\bf C1--C4}
	and let $( \P , \chi ,  \eta ,  \CC )\in\Uad$
	be fixed. 
	Then, the control-to-state operator 
	$\S:\erre^4 \to \X$ is \Frechet\ differentiable
	at $( \P , \chi ,  \eta ,  \CC )$, where
	\begin{align*}
	\X := \big( \H1 H \cap \L\infty V \big)
		\times
		\big( \L\infty H \cap \L2 V \big)^2.
	\end{align*}
	Moreover, for every increment $\bh\in\erre^4$, 
	we have $D\S( \P , \chi ,  \eta ,  \CC)[\bh]= (\xi,\nu,\z)$,
	where $(\xi,\nu,\z)$ is the unique solution to \linsys\ associated to $\bh$,
	obtained by Theorem \ref{THM:LIN}. 
\end{theorem}

\begin{proof}[Proof of Theorem \ref{THM:FRECHET}]
To begin with, let us recall that $\bh= (h_\P,h_{\chi},h_{\eta},h_\CC )$ and let us set
the corresponding control 
$(\P^{\bh},\chi^{\bh},\eta^{\bh},\CC^{\bh}):=(\P+h_\P,\chi+h_{\chi},\eta+h_{\eta},\CC+h_\CC )$.
Next, we denote by
\begin{align*}
	(\bph, \bm, \bs) &:= \S(\P,\chi,\eta,\CC),\\
	(\bhph, \bhmu, \bhs)& := \S(\P^{\bh},\chi^{\bh},\eta^{\bh},\CC^{\bh}),
	\\ 
	(\xi,\nu,\z)& := \hbox{Solution to the linearized system associated to $\bh$},
\end{align*}
and set
\begin{align*}
	\phi := \bhph - \bph - \xi, \quad
	\r := \bhmu - \bm - \nu, \quad
	\omega := \bhs - \bs - \z.	
\end{align*}
By comparing Theorem \ref{THM:WP:STRONG} 
with Theorem \ref{THM:LIN},
we deduce that the triplet $(\phi,\r,\omega)$ satisfies
\begin{align*}
	\phi  \in \H1 H \cap \L\infty V, \qquad
	 \r,\omega \in  \H1 \Vp \cap \L2 V.
\end{align*}
To prove the assertion, it is enough to show that
\begin{align*}
	\frac {\norma{\S(\P^{\bh},\chi^{\bh},\eta^{\bh},\CC^{\bh})
	-\S(\P,\chi,\eta,\CC)
	-(\xi,\nu,\zeta)
	}_{\X}}{|\bh|} 
	\to 0 \quad \hbox{as $|\bh|\to 0$,}
\end{align*}
where $|\bh|:=|\P^{\bh}|+|\chi^{\bh}|+|\eta^{\bh}|+|\CC^{\bh}|$.
By using the notation above, 
this amounts to showinig that
\begin{align}
	\label{fre:stima}
	\frac{\norma{(\phi,\r,\omega)}_\X}{|\bh|}	\to 0 \quad \hbox{as $|\bh|\to 0$,}
\end{align}
so that it suffices to check that there exist two constants $M>0$ and $\gamma>1$,
independent of $\bh$, such that 
\begin{align}
	\norma{(\phi,\r,\omega)}_\X \leq M |\bh|^\gamma.
	\label{fre:stima:2}
\end{align}
Taking the difference of the corresponding 
systems, we infer that the triplet $(\phi,\r,\omega)$
solves the following system
\begin{align}
	\label{fre:1}
	&\dt(\eps  \r+ \phi)- \Delta \r
	=
	L_\P
	&&\hbox{in $\, Q$},
	\\
	\label{fre:2}
	&\r
	=
	\tau\dt \phi
	+a \phi -J * \phi
	+  F''(\bph)\phi
	+L_{\chi}
 	&&\hbox{in $\,Q$},
	\\
	&\dt \omega
	- \Delta \omega
	+ \B \omega
	+L_\CC
	\label{fre:3}
	=
	L_{\eta}
	&&\hbox{in $\,Q$},
	\\
	&\dn \r=\eta\dn \phi=\dn \omega =0
 	&&\hbox{on $\,\Sigma$},
	\label{fre:4}
	\\ 
	&\r(0)= \phi(0)=\omega(0)=0
  	&&\hbox{in $\,\Omega,$}
	\label{fre:5}
\end{align}
where 
\begin{align*}
	L_\P&:= 
	\P f(\bph)\omega
	+ \P \big( f(\bhph)-f(\bph)\big) \big( \bhs-\bs \big)
	+ {(\P\bs-\A)}\big(	f(\bhph)-f(\bph)-f'(\bph)\xi	\big)
	\\ & \quad
	+ h_\P \big[	\big( f(\bhph)-f(\bph)\big) \big( \bhs-\bs \big)
	+  ( f(\bhph)-f(\bph)) \bs
	+( \bhs-\bs) f(\bph)
		\big],
	\\
	L_{\chi}&:= -\chi \omega - h_{\chi}(\bhs-\bs)
	+F'(\bhph) - F'(\bph) - F''(\bph)(\bhph-\bph),\\
	L_{\eta}&:= -\eta \Delta \phi - h_{\eta}\Delta(\bhph-\bph),
	\\
	L_\CC&:= 
	\CC f(\bph)\omega
	+ \CC \big( f(\bhph)-f(\bph)\big) \big( \bhs-\bs \big)
	+ \CC\big(	f(\bhph)-f(\bph)-f'(\bph)\xi	\big)\bs
	\\ & \quad
	+ h_\CC \big[	\big( f(\bhph)-f(\bph)\big) \big( \bhs-\bs \big)
	+  ( f(\bhph)-f(\bph)) \bs
	+( \bhs-\bs) f(\bph)
		\big].
\end{align*}
As an easy consequence of Theorem \ref{THM:WP:STRONG},
the following estimates hold:
\begin{align*}
	&\norma{\bph}_{W^{1,\infty}(0,T; V)\cap H^1(0,T; H^2(\Omega))} 
	+ \norma{\eta \bph}_{\L2 W}
	+ \norma{\dt\bph}_{L^\infty(Q)}
	\\ & \quad
	+ \norma{\bm }_{\H1 H \cap \L\infty V \cap \L2 W \cap L^\infty(Q)}
	+ \norma{\bs }_{\H1 H \cap \L\infty V \cap \L2 W\cap L^\infty(Q)}
	\leq K,
\end{align*}
and
\begin{align}
	& \non
	\norma{\bhph - \bph}_{W^{1,\infty}(0,T; V)\cap H^1(0,T; H^2(\Omega))}
	+ \norma{\bhmu-\bm}_{H^1(0,T; H)\cap L^\infty(0,T; V)\cap L^2(0,T; W)}
	\\ & \quad
	+ \norma{\bhs-\bs}_{H^1(0,T; H)\cap L^\infty(0,T; V)\cap L^2(0,T; W)}
	\leq K |\bh|,
	\label{fre:diff}
\end{align}
where the constant $K$ is independent of $\bh$.
In order to prove \eqref{fre:stima}, we multiply
\eqref{fre:1} by $\r$, \eqref{fre:2} to which we add to both sides 
$(c_a+2)\phi$ by $-\dt\phi$,
the gradient of \eqref{fre:2} by $\nabla \phi$, and \eqref{fre:3} by $\omega$.
Using the same argument employed in \eqref{lin:est:1}--\eqref{lin:est:2} in the proof of the 
linearized system, we deduce that,
upon integration over $Q_t$ and addition of the resulting equalities,
\begin{align}
	\non
	&\frac \eps 2 \IO2 {\r}
	+ \I2 {\nabla \r}
	+ \tau \I2 {\dt\phi}
	+\IO2 {\phi}
	+ C_0 \I2 {\nabla \phi}
	+\frac \tau 2 \IO2 {\nabla \phi}
	\\ & \qquad\non
	+ \frac 12 \IO2 {\omega}
	+\B \I2 { \omega}
	+ \I2 {\nabla \omega}
	\\ & \non
	\leq
	\intQt L_\P \r
	- \intQt F''(\bph)\phi \dt\phi
	- \intQt L_{\chi} \dt \phi
	- \intQt (c_a+2)\phi \dt \phi
	-\int_{Q_t}\nabla L_{\chi}\cdot\nabla\phi
	\\ & \qquad\label{est:fre}
	- \intQt \phi \nabla a \cdot \nabla \phi
	+ \intQt (\nabla J)* \phi \cdot \nabla \phi
	- \intQt F'''(\bph) \phi \nabla \bph \cdot \nabla \phi
	+ \intQt (L_\eta - L_\CC)\omega.
\end{align}
The second term on the \rhs\ can be estimated by using the separation property and Young's inequality which lead to
\begin{align*}
	\intQt F''(\bph)\phi \dt\phi 
	\leq M \intQt |\phi||\dt\phi|
	\leq \d \I2 {\dt\phi}
	+ \cd \I2 {\phi},
\end{align*}
for a positive constant $\d$ yet to be chosen.
Let us recall Taylor's formula with integral remainder for $f$, 
which is well-defined owing to the required regularity:
\begin{align}
	\label{taylor:f}
	f(\bhph)-f(\bph)-f'(\bph)\xi \,=\, f'(\bph) \phi + R_f^\bh (\phdiff)^2,
\end{align}
where the remainder $R_f^\bh$ is defined as
\begin{align*}
	R_f^\bh:= \int_0^1 f''(\bph+s (\bhph-\bph)) (1-s)\,\ds
\end{align*}
and it is uniformly bounded since $f\in C^2_b(\erre)$.
Using the Young and \Holder\ inequalities,
the boundedness and the \Lip\ continuity of $f$ and $f'$, the Taylor's formula \eqref{taylor:f},
the estimate \eqref{fre:diff},
as well as the regularity of $\bph$ and $\bs$, we have
\begin{align*}
	 &\intQt L_\P \r
	 - \intQt L_\CC \omega
	\\ & \quad
	\leq
	M \intQt (|\omega| + |\phdiff||\sdiff| + |\phi| + |\phdiff |^2 
	)(|\r|+|\omega|)
	\\ & \qquad
	+ M|\bh| \intQt (|\phdiff||\sdiff|+|\phdiff| + |\sdiff|)(|\r|+|\omega|)
	\\ &\quad \leq
	M \intQt (|\omega|^2 + |\r|^2 + |\phi|^2)
		+ M \iot \norma{\phdiff}_4^2(\norma{\r}	+\norma{\omega})
	\\ & \qquad
	+ M (1+|\bh|)\iot \norma{\phdiff}_4\norma{\sdiff}_4(\norma{\r}	+\norma{\omega})
	\\ & \qquad
	+ M|\bh| \iot (\norma{\phdiff}+\norma{\sdiff})(\norma{\r}	+\norma{\omega})
	\\ & \quad\leq
	M \intQt (|\omega|^2 + |\r|^2+ |\phi|^2)
	+ M (|\bh|^6 + |\bh|^4).
\end{align*}
Now, arguing again by Taylor's formula with integral remainder, we have that 
\[
F'(\bhph) - F'(\bph) - F''(\bph)(\bhph-\bph)=R_{F'}^\bh(\bhph-\bph)^2
\]
with remainder
\begin{align*}
	R_{F'}^\bh:= \int_0^1 F'''(\bph+s (\bhph-\bph)) (1-s)\,\ds
\end{align*}
which is bounded by a constant independent of $\bh$ by Corollary~\ref{COR:SEP}.
Taking these remarks into account,
we infer that
\begin{align*}
	&-\intQt L_{\chi} \dt \phi
	+ \intQt L_\eta \omega\\
	 & \quad \leq
	M \intQt |\omega||\dt\phi|
	+ M|\bh| \iot \norma{\sdiff}\norma{\dt\phi}
	+M\int_{Q_t}|\phdiff|^2|\partial_t\phi|
	\\ & \qquad
	+ M \intQt |\nabla  \phi||\nabla  \omega|
	+ M|\bh| \iot \norma{\Delta(\phdiff)}\norma{\omega}
	\\ 	&\quad \leq
	\d \intQt (| {\dt\phi}|^2 + |\nabla \omega|^2)
	\\&\qquad+\cd \intQt (|\omega|^2+ 
	|\bh|^2|\sdiff|^2 + |\phdiff|^4 + |\nabla \phi|^2 + |\bh|^2|\Delta(\phdiff)|^2)
	\\ 	&\quad \leq
	\d \intQt (| {\dt\phi}|^2 + |\nabla \omega|^2)
	+\cd \intQt (|\omega|^2+ |\nabla \phi|^2) + M(\delta)|\bh|^4.
\end{align*}
Furthermore, we have that 
\begin{align*}
	-\int_{Q_t}\nabla L_{\chi}\cdot\nabla\phi&=
	\chi \intQt \nabla \omega \cdot \nabla \phi
	+ h_{\chi}\intQt  \nabla (\sdiff) \cdot \nabla \phi\\
	& \quad -\int_{Q_t}\nabla(F'(\bhph) - F'(\bph) - 
	F''(\bph)(\bhph-\bph))\cdot\nabla\phi,
\end{align*}
where, by the Young inequality,
\begin{align*}
  \chi \intQt \nabla \omega \cdot \nabla \phi
	+ h_{\chi}\intQt  \nabla (\sdiff) \cdot \nabla \phi
 &\leq
  \d \I2 {\nabla \omega} 
  + \cd \intQt (|\nabla \phi|^2+ |\bh|^2|\nabla (\sdiff)|^2)\\
  &\leq \d \I2 {\nabla \omega} 
  + \cd \intQt |\nabla \phi|^2 + \cd|\bh|^4,
\end{align*}
and
\begin{align*}
  &\int_{Q_t}\nabla(F'(\bhph) - F'(\bph) - 
	F''(\bph)(\bhph-\bph))\cdot\nabla\phi\\
  &\quad=\int_{Q_t}(F''(\bhph)\nabla\bhph - F''(\bph)\nabla\bph - 
	F'''(\bph)\nabla\bph(\bhph-\bph) - F''(\bph)\nabla(\bhph-\bph))\cdot\nabla\phi\\
  &\quad=\int_{Q_t}(F''(\bhph)-F''(\bph)-F'''(\bph)(\bhph-\bph))\nabla\bph\cdot\nabla\phi
  +\int_{Q_t}(F''(\bhph)-F''(\bph))\nabla(\bhph-\bph)\cdot\nabla\phi.
\end{align*}
Hence, using again the Taylor formula with integral remainder for $F''$,
the separation property \eqref{separation:est},
and the estimate \eqref{fre:diff}, 
it is straightforward to see that 
\[
	-\int_{Q_t}\nabla(F'(\bhph) - F'(\bph) - 
	F''(\bph)(\bhph-\bph))\cdot\nabla\phi
	\leq M\int_{Q_t}|\nabla\phi|^2 + M|\bh|^4.
\]
Finally, the last line of \eqref{est:fre} can be bounded from above
using the H\"older inequality, the continuous embedding $V\hookrightarrow L^4(\Omega)$,
and the regularity of $\bph$
as
\begin{align*}
	& 
	\frac \tau 2 \I2 {\dt\phi}
	+ M \I2 \phi
	+ M \intQt (|\phi|^2+|\nabla \phi|^2)+
	M\int_0^t\|\nabla\bph\|_4^2\|\phi\|_4^2 
	\\
	& \quad \leq 
	\frac \tau 2 \I2 {\dt\phi}
	+ M\int_0^t\|\phi\|_V^2.
\end{align*}
Therefore, upon collecting the above estimates, picking $\d$ small enough, and invoking 
Gronwall's lemma, we conclude the proof since \eqref{fre:stima:2}
has been shown with $\gamma =2$.
\end{proof}

\subsection{Adjoint system}
In order to study first-order conditions for optimality for problem $(CP)_\ept$,
for a fixed admissible control $(\ov\P,\ov\chi,\ov\eta,\ov\CC)\in\Uad$
with corresponding state $(\bph,\bm,\bs)$,
we introduce and solve the auxiliary backward-in-time problem called 
{\it adjoint system}, in the new variables $(p,q,r)$. 
This system is formally obtained 
by taking the adjoint of the linearized system \eqref{lin:1}--\eqref{lin:5}, and reads
\begin{align}
	& \non
	-\dt (p +\tau q)
	+a q - J*q
	+\ov\eta \Delta r
	+ F''(\bph)q
	\\ & \hspace{3cm}
	+ \ov\CC\bs f'(\bph) r
	- (\ov\P\bs -\A)f'(\bph) p \,=\, 	\bQ (\bph-\phQ)
 	&& \hbox{in $\,Q$},
 	\label{adj:1}
	\\	
	\label{adj:2}
	& 	-\eps \dt p	- \Delta p 	- q 	\,=\,	0
	&& \hbox{in $\, Q$},
	\\
	\label{adj:3}
	& -\dt r
	- \Delta r
	+ (\B + \ov\CC f(\bph)) r
	- \ov\P f(\bph)p
	- \ov\ch q
	\,=0
	&&  \hbox{in $\,Q$},
	\\
	&\dn p\,=\,\dn r \,=\,0
 	&& \hbox{on $\,\Sigma$},
	\label{adj:4}
	\\ 
	& \eps p(T)\,=\,0, \quad (p+\tau q)(T)\,=\, \bO (\bph(T)-\phO), \quad r(T)\,=0
  	&& \hbox{in $\,\Omega.$}
	\label{adj:5}
\end{align}
\Accorpa\adjsys {adj:1} {adj:5}

Here we state the corresponding well-posedness result.
\begin{theorem}[Well-posedness of the adjoint system: $\ept >0$]
	\label{THM:ADJ:ept}
	Assume {\bf A1--A7} and {\bf C1--C4}
	and let $(\ov\P, \ov\chi , \ov\eta , \ov\CC )\in\Uad$
	be an admissible control, with corresponding state $(\bph,\bm,\bs)$.
	Then, the adjoint system \adjsys\ admits a unique solution $(p,q,r)$ such that
	\begin{align*}
		p,r  &\in \W {1,\infty} H \cap \H1 V \cap \L\infty W,
		\qquad
		q  \in \H1 H \cap \L\infty H.
	\end{align*}
\end{theorem}

\begin{proof}[Proof of Theorem \ref{THM:ADJ:ept}]
A rigorous proof has to be addressed within an approximation scheme.
Anyhow, since the system is linear and the arguments are standard
we just point out the formal a priori estimates, leaving the details to the reader.

\noindent
{\bf First estimate:}
We multiply \eqref{adj:1} by $q$, \eqref{adj:2} by $-\dt p + p $,
\eqref{adj:3} to which we add to both sides the term $r$ 
by $- \dt r$ and \eqref{adj:3} by $-\Delta r$.
After integrating over $Q_t^T$ and adding the resulting equalities, we obtain
\begin{align*}
	& 
	\frac {\tau} 2 \IO2 {q}
	+  C_0 \int_{Q_t^T}|q|^2
	+ \eps \int_{Q_t^T}|\dt p|^2
	+ \frac \eps2 \IO2 p
	+ \frac 12  \IO2 {\nabla p}
	+ \int_{Q_t^T}|\nabla p|^2
	\\ & \qquad
	+\frac {\B+1}2 \norma{r(t)}^2
	+ \IO2 {\nabla r} + \B\int_{Q_t^T}|\nabla r|^2
	+ \int_{Q_t^T}|\dt r|^2
	+  \int_{Q_t^T}|\Delta r|^2
	\\ & \leq
	\underbrace{\frac 1{2\tau} \norma{\bO (\bph(T)-\phO)}^2
	+ \intQtT  \bQ(\bph-\phQ)q
	- \intQtT \ov\eta \Delta r q
	+ \intQtT (J*q)q}_{=:I_1}
	\\ & \qquad
	\underbrace{- \intQtT  \ov\CC \bs f'(\bph)r q
	+\intQtT  (\ov\P\bs-\A) f'(\bph)p q
	+\intQtT q p
	+ \intQtT \ov\CC f(\bph) r (\dt r+\Delta r)}_{=:I_2}
	\\ & \qquad
	\underbrace{- \intQtT \P f(\bph)p (\dt r+\Delta r)
	- \intQtT \ov\chi q (\dt r+\Delta r)
	- \intQtT r \dt r}_{=:I_3}.
\end{align*} 
By virtue of the regularity of $\bs$, the boundedness of $f$ and $f'$
and Young's inequality, we easily infer that
\begin{align*}
	|I_1| & \leq
	\d \I2 {\Delta r} 
	 + \cd \intQtT (|q|^2+1),
	 \\
	 |I_2|+|I_3| & \leq
	 \d \intQtT (|\dt r|^2+| {\Delta r} |^2)
	 + \cd \intQtT (|p|^2+|q|^2+|r|^2)
\end{align*}
for a positive constant $\d$ yet to be chosen.
Hence, we take $\d$ small enough and Gronwall's lemma along with elliptic regularity, produces
\begin{align*}
	\norma{p}_{\H1 H \cap \L\infty V}
	+ \norma{q}_{\L\infty H }
	+\norma{r}_{\H1 H \cap \L\infty V \cap \L2W}
	\leq M.
\end{align*}

\noindent
{\bf Second estimate:}
A comparison argument in \eqref{adj:2} easily produce an $\L2 H$ bound for $\Delta p$.
Hence, using elliptic regularity theory we easily infer that
\begin{align*}
	\norma{p}_{\L2 W}
	\leq M.
\end{align*}

\noindent
{\bf Third estimate:}
From the above estimate, 
a comparison argument in \eqref{adj:1} leads us to obtain
\begin{align*}
	\norma{\dt q}_{\L2 H}
	\leq M.
\end{align*}

\noindent
{\bf Fourth estimate:}
Notice that \eqref{adj:2} and \eqref{adj:3}
have a parabolic structure in $p$ and $r$ with zero final condition and source term bounded in $\L\infty H$.
Therefore, it easily follows from classical parabolic regularity theory that
\begin{align*}
	\norma{p}_{\W {1,\infty} H \cap \H1 V \cap \L\infty W}
	+ \norma{r}_{\W {1,\infty} H \cap \H1 V \cap \L\infty W}
	\leq C.
\end{align*}

Arguing in a similar fashion as for the linearized system, due to the linearity of the adjoint system \adjsys, the uniqueness directly follows from the above estimates and the proof is concluded.
\end{proof}

\subsection{Optimality conditions}
This section is devoted to the study of necessary conditions for 
optimality for the optimization problem $(CP)_\ept$.

First of all, we employ a classical tool to derive first-order necessary conditions for $(CP)_\ept$.
In fact, provided that $\J$ is sufficiently smooth and recalling the structure of $\Uad$, 
a first-order necessary condition for $\opt\in\Uad$ to be optimal is to verify
the following variational inequality
\begin{align}
	\label{foc:abst}
	\< D\Jred \opt , (\P,\chi,\eta,\CC) - \opt> \geq 0 
	\quad \hbox{for every $(\P,\chi,\eta,\CC) \in \Uad$},
\end{align}
where $D \Jred$ denotes the G\^ateaux derivative of the {\it reduced} cost functional defined as
\begin{align*}
	\Jred (\P,\chi,\eta,\CC) := \J (\S_1(\P,\chi,\eta,\CC),\P,\chi,\eta,\CC), \quad
	(\P,\chi,\eta,\CC)\in\Uad.
\end{align*}
Theorem \ref{THM:FRECHET} allows us to exploit this result to obtain an explicit expression in terms of the linearized variables.

\begin{theorem}[First-order necessary condition for optimality]
\label{THM:OPT:FIRST}
Assume {\bf A1--A7} and {\bf C1--C4}, and
let $\opt$ be an optimal control  for problem $(CP)_\ept$,
with corresponding state $(\bph,\bm,\bs)$.
Then, $\opt$ necessarily satisfies
\begin{align}
	&  \iO\bO  (\bph(T)-\phO) \xi(T)+
	\non \intQ \bQ(\bph - \phQ) \xi
	+ \aP (\ov{\P}-\P_*) (\P - \ov{\P} \,)\\
	\non
	& \quad  + \achi (\ov{\chi}-\chi_*) (\chi - \ov{\chi} )
	+ \aeta (\ov{\eta}-\eta_*) (\eta - \ov{\eta} )
	+ \aC (\ov{\CC}-\CC_*) (\CC -\ov{\CC} ) \geq 0\\
	\label{foc:first}
	&\hbox{for every $(\P,\chi,\eta,\CC) \in \Uad$,}
\end{align}
where $\xi$ is the first component of the unique solution $(\xi,\nu,\z)$
to the linearized system obtained by Theorem \ref{THM:LIN} associated to
${\bh}=( \P - \ov{\P} ,\, \chi - \ov{\chi} , \, \eta - \ov{\eta} , \, \CC -\ov{\CC} )$.
\end{theorem}
\begin{proof}
By Theorem~\ref{THM:FRECHET} and the usual chain 
rule for \Frechet\--differentiable functions,
it follows immediately that the reduced cost functional 
$\Jred:\Uad\to\erre$ is Fr\'echet-differentiable at $\opt$.
Hence, the optimality of $\opt$ yields directly \eqref{foc:abst},
which in turn reads as \eqref{foc:first}.
\end{proof}

The next step consists in simplifying the necessary conditions 
for the minimizer presented above, by using the adjoint system.
\begin{theorem}[Final first-order necessary conditions for optimality]
\label{THM:OPT:FINAL}
Assume {\bf A1--A7} and {\bf C1--C4},
let $(\ov{\P}, \ov{\chi} , \ov{\eta} , \ov{\CC} )\in\Uad$
be an optimal control for $(CP)_\ept$, and let
$(\bph,\bm,\bs)$ and $(p,q,r)$ be the corresponding state and adjoint variables, respectively.
Then, $\opt$ necessarily verifies
\begin{align}
	\non
	& \intQ (\P - \ov{\P} \,) \bs f(\bph) p
	+ \intQ (\chi - \ov{\chi} ) \bs q
	- \intQ (\eta - \ov{\eta} ) \Delta \bph \, r
	- \intQ (\CC - \ov{\CC} ) \bs f(\bph) \, r
	\\ & \quad
	\non
	+ \aP (\ov{\P}-\P_*)(\P - \ov{\P} \,) 
	+ \achi (\ov{\chi}-\chi_*) (\chi - \ov{\chi} \,)\\
	\label{opt:final}
	&\quad+ \aeta (\ov{\eta}-\eta_*) (\eta - \ov{\eta} \,)
	+ \aC (\ov{\CC}-\CC_*)  (\CC - \ov{\CC} \,)
	\geq 0
	\qquad\text{for every $(\P,\chi,\eta,\CC) \in \Uad$.}
\end{align}
\end{theorem}
\begin{proof}
We note that \eqref{opt:final} directly follows from \eqref{foc:first}, provided to
show the identity
\begin{align}
&  \non\intQ \bQ (\bph-\phQ) \xi + \iO \bO (\bph(T)-\phO) \xi (T) 
\\ & \qquad =
\intQ h_\P \bs f(\bph) p
	+ \intQ h_{\chi} \bs q
	- \intQ h_{\eta} \Delta \bph \, r
	- \intQ h_\CC \bs f(\bph) \, r
	\label{proof:foc}
\end{align}
with $\bh=(\P- \ov{\P}, \chi - \ov{\chi}, \eta - \ov{\eta}, \CC- \ov{\CC}).$
To this end, we 
multiply \eqref{lin:1}--\eqref{lin:3} by $p,q,$ and $r$ in the order,
integrate over $Q$, and sum the equalities to obtain 
\begin{align*}
0& = \intQ p [\dt(\eps  \nu+ \xi)- \Delta \nu-(\ov \P \bs-\A) 
f'(\bph)\xi-\ov \P\z f(\bph)-h_\P \bs f(\bph) ]
\\ 
& \quad + \intQ q [- \nu + \tau\dt \xi
	+a \xi -J * \xi
	+ F''(\bph)\xi
	- \ov\ch \z
	- h_{\ch} \bs]
\\ 
& + \quad \intQ r [\dt \z - \Delta \z	+\B\z 
+ \ov \CC ( \z f(\bph) + \bs f'(\bph)\xi ) + h_\CC \bs f(\bph) + \ov\eta \Delta \xi+ h_{\eta} \Delta \bph].
\end{align*}
The terms involving the time derivatives can be easily handled by integrating by parts and using the initial conditions \eqref{lin:5} and the terminal conditions \eqref{adj:5} to obtain that
\begin{align*}
	 &\intQ p\dt(\eps  \nu+ \xi) 
	 +\intQ \tau q\dt \xi +\intQ \dt \z r \\
	 	 & \quad= - \intQ \eps \dt p \nu 
	 		- \intQ \dt p \xi- \intQ \tau\dt q \xi  
			- \intQ \dt r \z
	 		+ \iO (p + \tau q)(T)\xi(T)
	 	\\ &\quad = - \intQ \eps \dt p \nu - \intQ \dt(p+ \tau q) \xi 
		- \intQ \dt r \z
		+ \iO \bO(\bph(T)-\phO)\xi  (T).
\end{align*}
Moreover, integrating by parts and rearranging the terms we get
\begin{align*}
0& = \intQ \xi [ -\dt (p +\tau q)
	+a q - J*q
	+\ov\eta \Delta r
	+ F''(\bph)q
	+\ov \CC\bs f'(\bph) r
	- (\ov \P\bs {-\A})f'(\bph) p]
\\ 
 &\quad + \intQ \nu [-\eps \dt p	- \Delta p 	- q  ]
+\intQ \z [-\dt r
	- \Delta r
	+ (\B + \ov \CC f(\bph)) r
	- \ov \P f(\bph)p
	- {\ov \ch} q]
\\ & \quad
	-\intQ h_\P \bs f(\bph) p
	- \intQ h_{\chi} \bs q
	+\intQ h_{\eta} \Delta \bph \, r
	+ \intQ h_\CC \bs f(\bph) \, r
	+ \iO \bO(\bph(T)-\phO)\xi  (T).
\end{align*}
Hence, we recall the definition of the adjoint variables \eqref{adj:1}--\eqref{adj:3} to realize that the most part of the above terms simplify and the remaining equality is \eqref{proof:foc}, as we claimed. 
\end{proof}

\section{Asymptotic analysis}
\label{sec:asymp}
The goal of this section is to exploit the results established so far for $\CPept$ in the case $\ept >0$
to show that we can solve the optimal controls $\CPeps, \CPtau, \CPoo$
through asymptotic arguments.
In particular, we aim at passing to the limit as $\eps$ and $\tau$ go to zero, 
both separately and jointly,
in the optimality condition \eqref{opt:final}. 

As the asymptotic analysis for the state system has already been recalled in 
Theorems \ref{THM:PREV:ASY:epszero}, \ref{THM:PREV:ASY:tauzero} and \ref{THM:PREV:ASY:epstautozero} (see also \cite{SS}), 
the first novelty addressed here consists in understanding the asymptotic behavior 
of the adjoint system.
In this direction, we show that the adjoint variables,
depending on $\ept>0$, converge in some topology.
To this end, we begin with obtaining some 
uniform estimates with respect to $\ept$ so to pass to the limit using classical weak and weak-star
compactness arguments. 

The second step consists in approximate the optimal controls of $\CPeps, \CPtau, \CPoo$
by means of sequences of optimal controls of $\CPept$.
A combination of these steps will allow us to rigorously pass to the limit in the optimality conditions \eqref{opt:final}, recovering thus the corresponding ones for $\CPeps, \CPtau,$ and $\CPoo$.

\subsection{Uniform estimates on the adjoint problem}
\label{ssec:est_adj}
In this subsection, we assume to be in the setting
of either Theorem~\ref{THM:PREV:ASY:epszero}, or \ref{THM:PREV:ASY:tauzero}, or \ref{THM:PREV:ASY:epstautozero}.
For every $\eps\in(0,\eps_0)$ and $\tau\in(0,\tau_0)$,
let $(\P_\ept,\chi_\ept,\eta_\ept,\CC_\ept)\in\Uad$ be an admissible control,
and let $(\ov\ph_{\ept}, \ov\m_\ept,\ov\s_\ept)$ and $(p_\ept,q_\ept,r_\ept)$ denote the unique
solutions to the state system \statesys\ 
and the adjoint system \adjsys\ with $\ept>0$, respectively.

First of all, performing the same estimate as in the proof 
of Theorem~\ref{THM:ADJ:ept}, 
noting that $\{\ov\varphi_\ept\}_\ept$ is always
bounded in $C^0([0,T]; H)$ uniformly in both $\eps$ and $\tau$
thanks to Theorems~\ref{THM:PREV:ASY:epszero}, \ref{THM:PREV:ASY:tauzero} and \ref{THM:PREV:ASY:epstautozero},
we have that 
\begin{align}
	& \non
	\frac {\tau} 2 \IO2 {q_\ept}
	+  C_0 \int_{Q_t^T}|q_\ept|^2 
	+ \eps \int_{Q_t^T}|\dt p_\ept|^2
	+ \frac \eps2 \IO2 {p_\ept}
	+ \frac 12 \IO2 {\nabla p_\ept}
	+ \int_{Q_t^T} |{\nabla p_\ept}|^2
	\\ & \quad \non
	+\frac {\B+1}2 \IO2 {r_\ept}
	+ \IO2 {\nabla r_\ept} 
	+ \B  \int_{Q_t^T} | {\nabla r_\ept}|^2 	
	+ \int_{Q_t^T}|\dt r_\ept|^2
	+ \int_{Q_t^T}|\Delta r_\ept|^2
	\\ & \leq \non
	 M\left(\frac {\bO^2}{2\tau} + 1\right) + \delta\int_{Q_t^T}|q_\ept|^2
	- \intQtT \eta_\ept \Delta r_\ept q_\ept
	+M(\delta) \intQtT(|\bs_\ept p_\ept|^2+|\bs_\ept r_\ept|^2)
	\\ & \quad \non
	+ \int_{Q_t^T}(J*q_\ept)q_\ept
	+\delta'\int_{Q_t^T}(|\dt r_\ept|^2 + |\Delta r_\ept|^2)
	+ M(\delta,\delta')\intQtT (|r_\ept|^2 + |p_\ept|^2)
	\\ & \quad \label{est1_adj}
	- \intQtT \chi_\ept q_\ept (\dt r_\ept+\Delta r_\ept),
\end{align} 
where the constants $M$, $\delta$, $\delta'$, $M(\delta)$, and $M(\delta,\delta')$ 
are independent of 
$\eps$ and $\tau$.
The H\"older inequality and the continuous inclusion $V\hookrightarrow L^4(\Omega)$ yield also
\[
  M(\delta) \intQtT(|\bs_\ept p_\ept|^2+|\bs_\ept r_\ept|^2) \leq
  \cd \int_t^T\|\bs_\ept\|_V^2(\|p_\ept\|_V^2+\|r_\ept\|_V^2).
\]
Secondly, taking the mean of \eqref{adj:1} we have 
\begin{align*}
	&-\dt ((p_\ept)_\Omega +\tau (q_\ept)_\Omega)
	+ (F''(\bph_\ept)q_\ept)_\Omega + \CC_\ept(\bs_\ept f'(\bph_\ept) r_\ept)_\Omega\\
	&\qquad= ((\P_\ept\bs_\ept -\A)f'(\bph_\ept) p_\ept)_\Omega +\bQ( (\bph_\ept-\phQ))_\Omega,
\end{align*}
so that testing this latter by $(p_\ept)_\Omega +\tau (q_\ept)_\Omega$ we get
\begin{align}
  \non\frac12|(p_\ept(t))_\Omega +\tau (q_\ept(t))_\Omega|^2
  &\leq\frac{\bO^2}{2}|(\ov\varphi_\ept(T)-\varphi_\Omega)_\Omega|^2
  +\int_t^T\beta_Q^2\|\bph_\ept-\phQ\|_1^2\\
  &\quad\non
  +M(\delta)\int_t^T|(p_\ept)_\Omega +\tau (q_\ept)_\Omega|^2
  (1+\|F''(\bph_\ept)\|^2)+\delta\int_t^T\|q_\ept\|^2\\
  &\quad\label{est2_adj} +
  M\int_t^T(\|\bs_\ept r_\ept\|^2_1
  +\|\bs_\ept p_\ept\|^2_1+\| p_\ept\|^2_1)
\end{align}
where again the constants $M$ and $M(\d)$ are independent of $\eps$ and $\tau$.
Now, since $\tau\leq1$, by the Jensen inequality we have
\[
  \frac18|(p_\ept(t))_\Omega|^2\leq 
  \frac14|(p_\ept(t))_\Omega +\tau (q_\ept(t))_\Omega|^2
  +\frac{\tau}{4|\Omega|}\|q_\ept(t)\|^2,
\]
so that summing \eqref{est1_adj} and \eqref{est2_adj},
using again the fact that $\{\ov\varphi_\ept\}_\ept$ is always
bounded in $C^0([0,T]; H)$ uniformly in both $\eps$ and $\tau$,
and rearranging the terms,
by the Poincar\'e-Wirtinger inequality we infer that
 \begin{align}
	& \non
	m\tau \IO2 {q_\ept}
	+  C_0 \int_{Q_t^T}|q_\ept|^2
	+ \eps \int_{Q_t^T}|\dt p_\ept|^2
	+ \frac \eps2 \IO2 {p_\ept}
	+m\|p_\ept(t)\|_V^2
	\\ & \qquad \non
	+  \int_{Q_t^T} |\nabla p_\ept|^2
	+\frac 12\|r_\ept(t)\|_V^2
	+\B  \int_{Q_t^T} | {\nabla r_\ept}|^2
	+ \int_{Q_t^T}|\dt r_\ept|^2
	+ \int_{Q_t^T}|\Delta r_\ept|^2
	\\ & \leq \non
	 M\left(\frac {\bO^2}{2\tau} + 1\right) 
	 +\delta\int_{Q_t^T}|q_\ept|^2
	 +\delta'\int_{Q_t^T}(|\dt r_\ept|^2 + |\Delta r_\ept|^2)
	 +\int_{Q_t^T}(J*q_\ept)q_\ept\\
         & \qquad \non
         + M(\delta,\delta')\int_t^T (\|r_\ept\|^2 + \|p_\ept\|^2)
	 +\cd\int_t^T\|F''(\bph_\ept)\|^2(\|p_\ept\|^2+\tau\|q_\ept\|^2)\\
	&\qquad \non
	+\cd \int_t^T\|\bs_\ept\|_V^2(\|p_\ept\|_V^2+\|r_\ept\|_V^2)
	- \intQtT \eta_{\eps,\tau} \Delta r_\ept q_\ept
	\\ & \qquad \label{est3_adj}
	- \intQtT \chi_{\eps,\tau} q_\ept (\dt r_\ept+\Delta r_\ept),
\end{align} 
where $\delta,\delta'>0$ are arbitrary, and $m,M, M(\delta), M(\delta,\delta')>0$
are independent of $\eps$ and $\tau$.

\subsection{The optimization problem $(CP)_\tau$}
\label{ssec:eps0}
Here, we solve $\CPtau$ through an asymptotic approach by exploiting the proved results for $\CPept$ by letting $\eps \to 0$. Throughout the whole Section~\ref{ssec:eps0}, we 
assume the following framework:
\[
\tau\in(0,\tau_0) \text{ fixed,} \qquad \eta_{\rm max}=\alpha_\eta=0,
\qquad\text{\eqref{init_eps0}--\eqref{conv_init_eps0}--\eqref{bound_init_eps0}}.
\]
This means that we neglect the term in $\eta$ in the cost functional and in the state system,
implying that all the admissible controls are in the form $(\P,\chi,0,\CC)$
and that \eqref{coeff_eps0} is automatically satisfied:
with a slight abuse of notation, we look at $\Uad$ as a compact 
set in $\erre^3$, and use the symbol $(\P,\chi,\CC)$ for the generic admissible control in $\Uad$.

The first result that we present concerns existence of optimal controls for $\CPtau$.
\begin{theorem}
Assume {\bf A1--A8} and {\bf C1--C3}.
Then, the optimization problem $\CPtau$
admits a solution.
\end{theorem}
\begin{proof}
This result follows directly by adapting the direct method used in
the proof of Theorem~\ref{THM:EXISTECE:MIN}, taking into account the 
compactness of $\Uad$ and the convergence Theorem~\ref{THM:PREV:ASY:epszero}.
\end{proof}

After the existence is established, our main goal is to provide some necessary conditions for optimality by letting $\eps\to0$ in \eqref{opt:final} written with the subscripts $\ept$.
Let then $\redopttau\in\Uad$ be an optimal control 
for problem $(CP)_\tau$, and let $(\ov\ph_\tau, \ov\m_\tau,\ov\s_\tau)$
be the corresponding state variables solving \statesys\
with $\eps=0$, in the sense of Theorem~\ref{THM:PREV:ASY:epszero}.
Formally we expect that, as $\eps \to 0$, the optimality condition 
reads
\begin{align*}
	& \intQ (\P - \ov{\P}_\tau \,) \bs_\tau f(\bph_\tau) p_\tau
	- \intQ (\chi - \ov{\chi}_\tau ) \bs_\tau q_\tau
	- \intQ (\CC - \ov{\CC}_\tau ) \bs_\tau f(\bph_\tau) \, r_\tau
	\\ & \quad
	+ \aP (\ov{\P}_\tau-\P_*)(\P- \ov{\P}_\tau \,)
	+ \achi (\ov{\chi}_\tau -\chi_*)(\chi - \ov{\chi}_\tau \,)
	\\ & \quad
	+ \aC (\ov{\CC}_\tau-\CC_*)(\CC - \ov{\CC}_\tau \,)
	\geq 0 \quad \hbox{for every $(\P,\chi,\CC) \in \Uad$,}
\end{align*}
where $(p_\tau,q_\tau,r_\tau)$ stands for some
adjoint variables solving \adjsys\ with $\eps =0$, whose meaning is yet to be defined.
Unfortunately, the situation is slightly more delicate.
In fact, even if we prove that the adjoint variables 
 $(p_\ept,q_\ept,r_\ept)$ 
converge to some limit $(p_\tau,q_\tau,r_\tau)$ in a suitable sense as $\eps \to 0$, 
it is not obvious that every optimal control $\redopttau$
can be recovered as the limit of a sequence of optimal controls $\{\redoptept\}_\eps$ of $\CPept$.

To overcome this issue we follow the same line of argument of \cite{BARBU}(see also \cite{signori2,signori3,signori4} in the context of tumor growth models).
We introduce a different cost functional, called {\it adapted}, depending on
the fixed minimizer $\redopttau$ of $\CPtau$, which is defined as
\begin{align*}
	\Jad(\ph, \P, \chi, \CC) : = \J(\ph, \P, \chi, \CC)  
	+ \tfrac 12 |\P -  {\ov \P}_\tau|^2
	+ \tfrac 12 |\chi - {\ov \chi}_\tau|^2
	+ \tfrac 12 |\CC - {\ov \CC}_\tau|^2.
\end{align*}
Keeping the optimal control $\redopttau$ of $\CPtau$ fixed, note that
$\Jad \equiv \J$ on the minimizers of $\CPtau$.
The main idea behind this local perturbation concerns the fact that for the 
associated optimal control problem, which will be referred to as adapted, 
we can obtain a compactness-type property.
Namely, we prove that every arbitrary minimizer $\redopttau$ of $\CPtau$
can be recovered as limit 
of a sequence of minimizers of $\CPad$, as $\eps \to 0$.
The just mentioned adapted optimal control problem associated with $\ept$ reads as
\begin{align}
	\non
	\CPept^{\rm ad} \quad
	&\hbox{Minimize $\Jad(\ph,\P,\ch,\CC)$ subject to:}\\
	&\hbox{(i) $(\ph,\mu,\s)$ yields a solution to \statesys;}\non
	\\ 
	& \hbox{(ii) $(\P,\chi,\CC) \in \Uad.$}
	\label{CP:adapted}
\end{align}
In a sense to be made rigorous later, we will prove that 
$\CPept^{\rm ad}\searrow \CPtau$ so that
the passage to the limit as $\eps \to 0$ in the variational inequality \eqref{opt:final} can be rigorously performed producing in turn the optimality condition of $\CPtau$.
Since $\CPad$ fulfils the same assumptions of $\CPept$, 
for what we already proved in Section~\ref{sec:CP_ept} we readily infer the following.

\begin{lemma}
\label{LEM:EXISTENCE:ADAPTED}
Assume \ref{ass:1}--\ref{ass:8} and \ref{ass:op:1}--\ref{ass:op:3}.
Then, for every $\eps\in(0,\eps_0)$ and 
for every optimal control $({\ov \P}_\tau,{\ov \chi}_\tau,{\ov \CC}_\tau)\in\Uad$ of $(CP)_\tau$,
the optimization problem $\CPad$ admits a minimizer.
\end{lemma}

\begin{lemma}
\label{LEM:OPT:ADAPTED:EPS}
Assume  \ref{ass:1}--\ref{ass:7} and \ref{ass:op:1}--\ref{ass:op:4}, and let
$({\ov \P}_\tau,{\ov \chi}_\tau,{\ov \CC}_\tau)\in\Uad$
be an optimal control for $(CP)_\tau$.
For every $\eps\in(0,\eps_0)$, if
$({\ov \P}_\ept,{\ov \chi}_\ept,{\ov \CC}_\ept)\in\Uad$
is an optimal control for $\CPad$,
then the following first-order necessary condition holds
\begin{align}
	&\non \intQ (\P - \ov{\P}_\ept \,) \bs_\ept f(\bph_\ept) p_\ept
	- \intQ (\chi - \ov{\chi}_\ept ) \bs_\ept q_\ept
	- \intQ (\CC - \ov{\CC}_\ept ) \bs_\ept f(\bph_\ept) \, r_\ept
	\\ & \quad  \non
	+ (\P - \ov{\P}_\ept \,)\big( \aP (\ov{\P}_\ept-\P_*) + (\ov{\P}_\ept -\ov{\P}_\tau)\big)
	+(\chi - \ov{\chi}_\ept \,)\big( \achi (\ov{\chi}_\ept-\chi_*) + (\ov{\chi}_\ept-\ov{\chi}_\tau)\big)
	\\ & \quad
	+ (\CC - \ov{\CC} _\ept\,)\big( \aC (\ov{\CC}_\ept-\CC_*) + (\ov{\CC} _\ept-\ov{\CC}_\tau) \big)
	\geq 0
	\qquad
	\hbox{for every $(\P,\chi,\CC) \in \Uad$,}
	\label{opt:ad:epstozero}
\end{align}
where $(\bph_\ept,\bm_\ept,\bs_\ept)$ and $(p_\ept,q_\ept, r_\ept)$ denote the corresponding unique solutions to \statesys\ and \adjsys\ with $\ept>0$.
\end{lemma}

The sense in which the minimizers of $\CPad$ approximate the ones of $\CPtau$
as $\eps \to 0$ is specified in the following theorem.

\begin{theorem}
\label{THM:CONV:EPS:ADAPTED}
Assume \ref{ass:1}--\ref{ass:8} and \ref{ass:op:1}--\ref{ass:op:4}.
Let $\redopttau\in\Uad$ be an optimal control for $\CPtau$,
with corresponding state $(\ov\varphi_\tau, \ov\mu_\tau, \ov\sigma_\tau)$.
Then, for every family
$\{( {\ov \P}_{\ept},{\ov \chi}_{\ept},{\ov \CC}_{\ept})\}_\eps$
of optimal controls for $(CP)^{\rm ad}_{\ept}$,
with corresponding states
$\{(\bph_{\ept},\bm_{\ept},\bs_{\ept})\}_\eps$,
as $\eps\to0$ it holds that
\begin{align}
	& \non
	\bph_{\ept} \to \bph_\tau \quad \hbox{weakly-$^*$ in $\H1 H \cap \L\infty V,$}
	\\ & \hspace{3cm}
	\hbox{and strongly in $\C0 H,$}
	\label{adapted:conv:1}
	\\ \label{adapted:conv:2}
	&{\ov \P}_{\ept} \to {\ov \P}_\tau, \quad 
	{\ov \chi}_{\ept} \to {\ov \chi}_\tau,
	\quad
	{\ov \CC}_{\ept} \to {\ov \CC}_\tau,
	\\
	& \label{adapted:conv:3}
	\Jad (\bph_{\ept},{\ov \P}_{\ept},
	{\ov \chi}_{\ept},{\ov \CC}_{\ept}) 
	\to \J(\bph_\tau, {\ov \P}_\tau, {\ov \chi}_{\tau},{\ov \CC}_{\tau}).
\end{align}
\end{theorem}

\begin{proof}
Since $\Uad$ is a compact subset of $\erre^3$,
by virtue of \eqref{epstozero:1}, \eqref{epstozero:strong:1},
and the Bolzano--Weierstrass theorem,
we infer the existence of $\hat \ph\in \H1 H \cap \L\infty V$ and 
$(\hat \P,\hat\chi,\hat \CC) \in \Uad$ such that, 
along a non-relabelled zero subsequence $\eps_k$, as $k\to\infty$,
\begin{align*}
	& \bph_{\eps_k,\tau} \to \hat\ph
	\quad \hbox{weakly-$^*$ in $\H1 H \cap \L\infty V,$ and strongly in $\C0 H,$}
	\\ &{\ov \P}_{\eps_k,\tau} \to \hat \P, \quad 
	{\ov \chi}_{\eps_k,\tau} \to\hat\chi,
	\quad
	{\ov \CC}_{\eps_k,\tau} \to \hat \CC.
\end{align*}
By Theorem~\ref{THM:PREV:ASY:epszero}, we infer that
$\hat\varphi$ is actually the first component of the state system \statesys\
with $\eps=0$ and parameters $(\hat \P, \hat\chi, \hat \CC)$.
Now, on the one hand the minimality of 
$( {\ov \P}_{\eps_k,\tau},{\ov \chi}_{\eps_k,\tau},{\ov \CC}_{\eps_k,\tau})$ 
for $(CP)_{\eps_k,\tau}^{\rm ad}$ entails that
\begin{align*}
	\Jad (\bph_{\eps_k,\tau}, {\ov \P}_{\eps_k,\tau},{\ov \chi}_{\eps_k,\tau},{\ov \CC}_{\eps_k,\tau})
	\leq 
	\Jad (\bph_{\tau}, {\ov \P}_\tau,{\ov \chi}_\tau,{\ov \CC}_\tau)
	= \J (\bph_{\tau}, {\ov \P}_\tau,{\ov \chi}_\tau,{\ov \CC}_\tau),
\end{align*}
so that passing to the superior limit to both sides leads to
\begin{align*}
	\limsup_{k \to \infty}
	\Jad (\bph_{\eps_k,\tau}, {\ov \P}_{\eps_k,\tau},{\ov \chi}_{\eps_k,\tau},{\ov \CC}_{\eps_k,\tau})
	\leq 
	\J (\bph_{\tau}, {\ov \P}_\tau,{\ov \chi}_\tau,{\ov \CC}_\tau).
\end{align*}
On the other hand, by the lower semincontinuity of $\Jad$, we also have that
\begin{align*}
	&\liminf_{k\to\infty} 
		\Jad(\bph_{\eps_k,\tau}, {\ov \P}_{\eps_k,\tau},
				{\ov \chi}_{\eps_k,\tau},{\ov \CC}_{\eps_k,\tau})
	\geq 
	 \Jad (\hat\ph, \hat \P,\hat\chi,\hat \CC) 
	\\&  \quad
	 =  \J (\hat\ph, \hat \P,\hat\chi,\hat \CC) 
	 + \tfrac 12 (|\hat \P - {\ov \P}_\tau|^2+|\hat \chi - {\ov \chi}_\tau|^2+|\hat \CC - {\ov \CC}_\tau|^2).
\end{align*}
Since $\hat\varphi$ is the first state component of the system with $\eps=0$
and coefficients $(\hat \P, \hat\chi, \hat \CC)$,
combining the above inequalities with the
optimality of $(\bph_{\tau}, {\ov \P}_\tau,{\ov \chi}_\tau,{\ov \CC}_\tau)$ for $\CPtau$ yields
directly
\[
\hat \P = {\ov \P}_\tau, \quad
	\hat \chi = {\ov \chi}_\tau, \quad
	\hat \CC= {\ov \CC}_\tau,
\]
from which also $\hat \ph= \bph_\tau$ by uniqueness of 
the state system \statesys\ with $\eps=0$.
Also, we have the chain of equalities
\begin{align*}
	& \lim_{k\to\infty} \Jad (\bph_{\eps_k,\tau}, {\ov \P}_{\eps_k,\tau},
				{\ov \chi}_{\eps_k,\tau},{\ov \CC}_{\eps_k,\tau})
				= \liminf_{k\to\infty}\Jad (\bph_{\eps_k,\tau}, {\ov \P}_{\eps_k,\tau},
				{\ov \chi}_{\eps_k,\tau},{\ov \CC}_{\eps_k,\tau})
				\\ & \quad
				= \limsup_{k\to\infty}
				\Jad (\bph_{\eps_k,\tau}, {\ov \P}_{\eps_k,\tau},
				{\ov \chi}_{\eps_k,\tau},{\ov \CC}_{\eps_k,\tau})
				= \J (\bph_{\tau}, {\ov \P}_\tau,{\ov \chi}_\tau,{\ov \CC}_\tau).
\end{align*}
As the same argument holds along every arbitrary subsequence $\{\eps_k\}_k$,
by uniqueness of the limits the convergences actually hold along the 
whole sequence $\eps$,
and the proof is concluded.
\end{proof}

\subsubsection{Letting $\eps\to0$ in the adjoint system}
\label{SUB:ADJ:EPSTOZERO}
This section is devoted to discuss and analyze 
the asymptotic behavior of the adjoint system
\adjsys\ as $\eps \to 0$, which will be a key ingredient 
to derive the optimality conditions of $\CPtau$. 
To begin with, let us state the established result.

\begin{theorem}
	\label{THM:ADJ:EPSTOZERO}
	Assume \ref{ass:1}--\ref{ass:8} and \ref{ass:op:1}--\ref{ass:op:4}.
	Let $(\P_\tau,\chi_\tau,\CC_\tau)\in\Uad$, $\{(\P_\ept,\chi_\ept,\CC_\ept)\}_\eps\subset\Uad$
	be such that $(\P_\ept, \chi_\ept, \CC_\ept)\to(\P_\tau,\chi_\tau,\CC_\tau)$ as $\eps\to0$.
	Let $(\bph_\tau, \bm_\tau,\bs_\tau)$ and $(\bph_\ept, \bm_\ept,\bs_\ept)$
	be the unique solutions to the state system \statesys\
	in the cases $\eps=0$ with coefficients $(\P_\tau,\chi_\tau,\CC_\tau)$ and
	$\eps\in(0,\eps_0)$ with coefficients $(\P_\ept,\chi_\ept,\CC_\ept)$,
	as given by Theorems~\ref{THM:PREV:ASY:epszero} and \ref{THM:WP:STRONG},
	respectively.
	Let also $(p_\ept,q_\ept,r_\ept)$ be the unique solution to the adjoint system 
	\adjsys\ with $\eps\in(0,\eps_0)$ and coefficients $(\P_\ept, \chi_\ept, \CC_\ept)$,
	as given by Theorem \ref{THM:ADJ:ept}.
	Then, there exists a triplet $(p_\tau,q_\tau,r_\tau)$, with 
	\begin{align*}
	&p_\tau \in L^\infty(0,T; V)\cap L^2(0,T; W),\qquad q_\tau\in L^\infty(0,T; H),\qquad
	p_\tau+\tau q_\tau \in H^1(0,T; V^*),\\
	&r_\tau\in\H1 H \cap \L\infty V \cap \L2 W,
	\end{align*}
	such that, for every $\alpha \geq 1$ if $d=2$ and $1\leq \alpha <6$ if $d=3$, 
	as $\eps\to0$ it holds 
	\begin{align*}
	p_{\eps,\tau} & \to p_\tau && \hbox{weakly-$^*$ in $\L\infty V\cap \L2 W$},
	\\
	q_{\eps,\tau} & \to q_\tau && \hbox{weakly-$^*$ in $\L\infty H$},
	\\
	r_{\eps,\tau} & \to r_\tau && \hbox{weakly-$^*$ in $\H1 H \cap \L\infty V \cap \L2 W$},
	\\ & && \quad \hbox{strongly in $\C0 {\Lx \alpha} \cap \L2 V$,}
	\\
	p_{\eps,\tau} + \tau q_{\eps,\tau} & \to p_\tau + \tau q_\tau
		&& \hbox{weakly in $\H1 {\Vp}$},
	\\
	\eps p_{\eps,\tau} & \to 0
	&& \hbox{strongly in $\H1 H$}.
	\end{align*}
	Moreover, $(p_\tau,q_\tau,r_\tau)$ is the unique
	weak solution to the adjoint system \adjsys\ with $\eps =0$
	and coefficients $(\P_\tau,\chi_\tau,\CC_\tau)$,
	in the sense that
	\begin{align*}
	& \non
	-\<\dt (p_\tau +\tau q_\tau), v>_{V}
	+ \iO (a q_\tau - J*q_\tau )v
	+ \iO  F''(\bph_\tau)q_\tau v
	\\ & \qquad
	+ \iO  \CC_\tau\bs_\tau f'(\bph_\tau) r_\tau v
	- \iO  \P_\tau\bs_\tau f'(\bph_\tau) p_\tau v
	=\,\iO  	\bQ (\bph_\tau-\phQ) v,
	\\	
	& 	
	\iO \nabla p_\tau\cdot \nabla  w
	- \iO q_\tau w	\,=\,	0,
	\\
	& -\iO \dt r_\tau z
	+ \iO \nabla r_\tau \cdot \nabla z
	+ \iO \CC_\tau f(\bph_\tau) r_\tau z
	- \iO \P_\tau f(\bph_\tau)p_\tau z 
	- \iO \ch_\tau q_\tau z 
	\,=0,
\end{align*}
for every $v,w,z  \in V$, almost everywhere in $(0,T)$, and
\[
	(p_\tau + \tau q_\tau)(T)\,=\, \bO (\bph_\tau(T)-\phO), \qquad r_\tau(T)\,=0.
\]
\end{theorem}

\begin{proof}[Proof of Theorem \ref{THM:ADJ:EPSTOZERO}]
We use the estimate \eqref{est3_adj}.
First of all,
by \ref{ass:8} and Theorem~\ref{THM:PREV:ASY:epszero}, we have that
$\{F''(\bph_\ept)\}_\eps$ is uniformly bounded in $L^\infty(0,T; L^3(\Omega))$
and $\{\bs_\ept\}_\eps$ is uniformly bounded in $L^2(0,T;V)$: hence,
recalling that $\tau$ is fixed and $\eta=0$,
by the Gronwall lemma along with elliptic regularity theory,
there exists a positive constant $M_\tau$, which may depend on $\tau$ 
but is independent of $\eps$, such that
\begin{align*}
	&\eps^{1/2} \norma{p_{\ept}}_{\H1 H}
	+ \norma{p_{\ept}}_{\L\infty V}
	+ \norma{q_{\ept}}_{\L\infty H}\\
	&\qquad+ \norma{r_{\ept}}_{\H1 H \cap \L\infty V \cap L^2(0,T; W)}
	\leq M_\tau.
\end{align*}
In particular, by the H\"older inequality it follows that
\[
  \norma{F''(\bph_\ept)q_\ept}_{L^\infty(0,T; L^{6/5}(\Omega))} \leq M_\tau.
\]
Secondly, elliptic regularity theory and equation \eqref{adj:2} entail that
\begin{align*}
	\norma{p_\ept}_{\L2 W}
	\leq M_\tau.
\end{align*}
Moreover, since $L^{6/5}(\Omega)\hookrightarrow V^*$, 
a comparison argument in \eqref{adj:1} yields,
by the boundedness of $\{\bs_\ept\}_\eps$ in $L^\infty(Q)$
and the estimates above,
that 
\[
  \norma{p_\ept+\tau q_\ept}_{H^1(0,T; V^*)} \leq M_\tau.
\]
The Banach--Alaoglu theorem and classical compact embedding results (see, e.g., \cite{Simon})
allow us to obtain from the above a priori estimates 
the existence of functions $(p_\tau,q_\tau,r_\tau)$ such that as $\eps \to 0$
it holds, for every $\alpha \geq 1$ if $d=2$ and $1\leq \alpha <6$ if $d=3$, and along a non-relabelled subsequence,
\begin{align*}
	p_\ept & \to p_\tau  &&\hbox{weakly-$^*$ in $\L\infty V \cap \L2 W$},
	\\
	q_\ept & \to q_\tau  &&\hbox{weakly-$^*$ in $\L\infty H$},
	\\
	r_\ept & \to r_\tau  &&\hbox{weakly-$^*$ in $\H1 H \cap \L\infty V \cap \L2 W$},
	\\ & && \quad \hbox{strongly in $\C0 {\Lx \alpha} \cap \L2 V$,}
	\\
	p_\ept + \tau q_\ept & \to p_\tau + \tau q_\tau
		 &&\hbox{weakly in $\H1 {\Vp}$},
	\\
	\eps p_\ept & \to 0
		 &&\hbox{strongly in $\H1 H \cap L^\infty(0,T; V)\cap \L2 W$ }.
\end{align*}
We claim that these limit variables yield a weak solution to the 
adjoint system \adjsys\ in which we formally set $\eps=0$.
In this direction, we just need to justify the passage to the limit as $\eps \to 0$
in the variational formulation for system \adjsys\ written for the triplet $(p_\ept,q_\ept,r_\ept)$,
which reads
\begin{align}
	& \non
	- \iO \dt (p_\ept +\tau q_\ept) v
	+ \iO (a q_\ept - J*q_\ept )v
	+ \iO  F''(\bph_\ept)q_\ept v
	\\ & \quad
	+ \iO  \CC_\ept\bs_\ept f'(\bph_\ept) r_\ept v
	- \iO  \P_\ept\bs_\ept f'(\bph_\ept) p_\ept v
	=\,\iO  	\bQ (\bph_\ept-\phQ) v,
 	\label{adj:vf:1}
	\\	
	\label{adj:vf:2}
	& 	
	- \iO \eps \dt p_\ept w	
	+ \iO \nabla p_\ept \cdot \nabla w
	- \iO q_\ept w	\,=\,	0,
	\\
	\label{adj:vf:3}
	& -\iO \dt r_\ept z
	+ \iO \nabla  r_\ept \cdot \nabla z
	+ \iO \CC_\ept f(\bph_\ept) r_\ept z
	- \iO \P_\ept f(\bph_\ept)p_\ept z 
	- \iO \ch_\ept q_\ept z 
	\,=0,
\end{align}
for every $v , w,z \in V$ and almost every $t \in (0,T)$
and also in the terminal conditions
\begin{align}
	\label{adj:vf:4}
	\eps p_\ept(T)\,=\,0, \quad (p_\ept+\tau q_\ept)(T)\,=\, \bO (\bph_\ept(T)-\phO), \quad r_\ept(T)\,=0.
\end{align}
It is worth noting that since $\eps >0$ the second condition of \eqref{adj:vf:4}
reduces to $\tau q_\ept(T)= \bO (\bph_\ept(T)-\phO)$.
Moreover, from the convergences
 \eqref{epstozero:strong:1}--\eqref{epstozero:strong:2}, we also have that,
possibly after another extraction,
\begin{align}
    \label{strong:eps:1}
	\bph_\ept & \to \bph_\tau \quad \hbox{strongly in $\C0 {\Lx \alpha}$,
	 and $a.e.$ in $Q$},
	\\ 
	  \label{strong:eps:2}
	 \bs_\ept &\to \bs_\tau \quad\text{ strongly in } C^0([0,T]; V^*)\cap L^2(0,T; {\Lx \alpha})\, ,
\end{align}
for every $\alpha \geq 1$ if $d=2$ and $1\leq \alpha <6$ if $d=3$.
Therefore, most part of the above limits are
easy consequence of \eqref{epstozero:strong:1}--\eqref{epstozero:strong:2},
the above estimates and Lebesgue's dominated convergence theorem as well.
For instance, due to the boundedness and continuity of $f'$ we have, e.g., that
$f'(\bph_\ept) \to f'(\bph_\tau)$ $a.e. $ in $Q$ so that
we easily infer that
\begin{align*}
	\int_\Omega  \CC_\ept\bs_\ept f'(\bph_\ept) r_\ept v
	\to 
	\int_\Omega  \CC_\tau\bs_\tau f'(\bph_\tau) r_\tau v
	\quad \hbox{for every $v \in V,$}
\end{align*}
and the other terms can be handled in a similar fashion.
The only term which has to be treated differently is the one involving the potential.
It can be dealt with invoking the almost everywhere convergence \eqref{strong:eps:1},
the weak-$^*$ convergence \eqref{epstozero:1}, 
the continuous embedding $V \hookrightarrow \Lx6$,
and the Severini-Egorov theorem. 
In fact, these properties imply in particular that, as $\eps \to 0$,
\begin{align*}
		F''(\bph_\ept) \to F''(\bph_\tau) \quad
		\hbox{strongly in $L^\beta(Q)$} \quad \hbox{for all $\beta\in[1,3)$}
\end{align*}
so that from the weak-strong convergence principle, we get
\begin{align*}
		F''(\bph_\ept)q _\ept \to F''(\bph_\tau) q _\tau
		\quad
		\hbox{weakly in $\L2 {L^{\gamma}(\Omega)}$ for all $\gamma\in[1,6/5)$}.
\end{align*}
It follows then that 
\[
  \iO  F''(\bph_\ept)q_\ept v \to
  \iO  F''(\bph_\tau)q_\tau v \qquad\forall\,v\in W.
\]
This is enough to pass to the limit in the variational formulation 
\eqref{adj:vf:1}--\eqref{adj:vf:3} as $\eps \to 0$
for every $v\in W$, $w,z  \in V$,
and to obtain the required terminal conditions.
Since at the limit $F''(\bph_\tau)q_\tau \in L^\infty(0,T; L^{6/5}(\Omega))$,
by the density of $W$ in $V$ the variational formulation holds also for all $v\in V$.
Thus, we realize that the limit variables obtained above
yield a weak solution to \adjsys\ in which $\eps$ is set to zero.
By linearity and the estimate \eqref{est3_adj}, we deduce that 
$(p_\tau, q_\tau, r_\tau)$ is the unique weak solution to \adjsys\ with $\eps=0$,
hence also that the convergences above hold along the entire sequence $\eps\to0$,
and the proof is concluded.
\end{proof}

\subsubsection{Letting $\eps\to0$ in the optimality condition}
In this last step, we draw some consequences from the approximation of controls presented in Theorem \ref{THM:CONV:EPS:ADAPTED} and Subsection \ref{SUB:ADJ:EPSTOZERO}
by passing to the limit in the variational inequality \eqref{opt:ad:epstozero}
as $\eps \to 0$. 
This allows to prove the optimality conditions of $\CPtau$ as follows:
\begin{theorem}
\label{THM:AD:OPT:EPS}
Assume \ref{ass:1}--\ref{ass:8} and \ref{ass:op:1}--\ref{ass:op:4}.
Then, every optimal control $({\ov \P}_{\tau},{\ov \chi}_{\tau},{\ov \CC}_{\tau})$ of $\CPtau$
necessarily verifies
\begin{align*}
	& \intQ (\P - \ov{\P}_\tau \,) \bs_\tau f(\bph_\tau) p_\tau
	- \intQ (\chi - \ov{\chi}_\tau ) \bs_\tau q_\tau
	- \intQ (\CC - \ov{\CC}_\tau ) \bs_\tau f(\bph_\tau) \, r_\tau
	\\ & \quad
	+ \aP (\ov{\P}_\tau-\P_*)(\P- \ov{\P}_\tau \,)
	+ \achi (\ov{\chi}_\tau -\chi_*)(\chi - \ov{\chi}_\tau \,)
	+ \aC (\ov{\CC}_\tau-\CC_*)(\CC - \ov{\CC}_\tau \,)
	\geq 0\\
	&\quad \hbox{for every $(\P,\chi,\CC) \in \Uad$,}
\end{align*}
where $(\bph_\tau, \bm_\tau,\bs_\tau)$
and $(p_\tau,q_\tau,r_\tau)$ are the unique solutions 
to \statesys\ and \adjsys\ with $\eps =0$
in the sense of Theorems \ref{THM:PREV:ASY:epszero} and \ref{THM:ADJ:EPSTOZERO}, respectively.
\end{theorem}
\begin{proof}
Let $\{( {\ov \P}_{\eps,\tau},{\ov \chi}_{\eps,\tau},{\ov \CC}_{\eps,\tau})\}_\eps$
be an approximating sequence of minimizers for $\CPad$,
as given by Lemma~\ref{LEM:EXISTENCE:ADAPTED}.
Then, by Lemma~\ref{LEM:OPT:ADAPTED:EPS} the corresponding states
$\{(\bph_{\eps,\tau},\bm_{\eps,\tau},\bs_{\eps,\tau})\}_\eps$
and adjoint variables $\{(p_{\eps,\tau},q_{\eps,\tau},r_{\eps,\tau})\}_\eps$
satisfy \eqref{opt:ad:epstozero}.
By the convergences in
Theorems~\ref{THM:PREV:ASY:epszero}, 
\ref{THM:CONV:EPS:ADAPTED}, and \ref{THM:ADJ:EPSTOZERO}, 
the thesis follows letting $\eps\to0$ in \eqref{opt:ad:epstozero}
using the dominated convergence theorem.
\end{proof}

\subsection{The optimization problem $(CP)_\eps$}
\label{ssec:tau0}
Here, we continue the asymptotic analysis of the optimization problem $(CP)_\ept$,
focusing on the case $\tau \to 0$, and keeping $\eps$ fixed instead. 
Namely, throughout the whole Section~\ref{ssec:tau0} we assume
the following framework:
\begin{align*}
  &\eps\in\left(0,\eps_0\right), \qquad \bO=0, \qquad
  \text{\eqref{init_tau0}--\eqref{bound_init_tau0}}\\
  &\chi_{\rm max} < \sqrt{c_a}, \qquad
   (\chi_{\rm max}+\eta_{\rm max}+4c_a\chi_{\rm max})^2<8c_aC_0,
   \qquad\eta_{\rm max}^2+\chi_{\rm max}^2<\frac{4}{9}C_0.
\end{align*}
This implies that every admissible control $(\P,\chi,\eta,\CC)\in\Uad$
automatically satisfies \eqref{coeff_tau0}.
Moreover, let us remark that the assumption $\bO=0$ is rather unpleasant since
it prevents us to control the tumor distribution at the terminal time $T$. 
However, from \eqref{adj:5} it is clear that this 
compatibility condition has to be imposed in the scenario $\eps>0$ and $\tau=0$.

Existence of optimal controls for $\CPeps$ is given in the following result.
\begin{theorem}
Assume {\bf A1--A7} and {\bf C1--C3}.
Then, the optimization problem $\CPeps$
admits a solution.
\end{theorem}
\begin{proof}
This result follows directly by adapting the direct method used in
the proof of Theorem~\ref{THM:EXISTECE:MIN}, taking into account the 
compactness of $\Uad$ and the convergence presented in Theorem~\ref{THM:PREV:ASY:tauzero}.
\end{proof}

The main goal is to obtain now necessary conditions for optimality 
of $(CP)_\eps$.
The idea is to proceed on the same line of Section~\ref{ssec:eps0},
by passing to the limit as $\tau\searrow0$ in 
the adjoint problem and in the first 
order conditions for optimality for the adapted optimization problem $\CPad$. 
The main difference with respect to the case
$\eps\searrow0$ is that the state system \statesys\ with $\tau=0$
does not have a unique solution, as stated in Theorem~\ref{THM:PREV:ASY:tauzero}.
Notice that also uniqueness has been established in \cite{SS} as a consequence of a suitable
error estimate between the solutions to the $\ept$ and $\eps$ problems.
However, that result forces the authors to restrict to assume $\eta=0$.
This means that under no additional requirements on the data
(in particular, if $\eta_{\rm max}>0$),
the control-to-state map is not even well-defined when $\tau=0$.
The main problem is that,
despite Theorem~\ref{THM:PREV:ASY:tauzero}, 
for a given minimizer 
$(\varphi_\eps,\P_\eps,\chi_\eps,\eta_\eps,\CC_\eps)$
of $(CP)_\eps$, it is not necessarily true that the state $\varphi_\eps$ can be approximated 
by some corresponding state solutions $\{\varphi_\ept\}_\tau$ 
of the state system with $\eps,\tau>0$
as $\tau\searrow0$. For this reason,
it is important that the adapted cost functional is modified 
accordingly, accounting also for the phase variable.
Namely, in this section,
given a certain minimizer $(\ov\ph_\eps, \ov \P_\eps, \ov\chi_\eps, \ov\eta_\eps, \ov \CC_\eps)$
for $(CP)_\eps$, 
we consider the following adapted cost functional:
\begin{align*}
	 \Jad(\ph, \P, \chi, \eta, \CC) &:= \J(\ph, \P, \chi,\eta, \CC)  
	\!+\tfrac12 \norma{\varphi - \ov\ph_\eps}^2_{L^2(Q)} 
	\!+ \tfrac 12 |\P -  {\ov \P}_\eps|^2
	\\& \quad
	\!+ \tfrac 12 |\chi - {\ov \chi}_\eps|^2
	\!+ \tfrac 12 |\eta - {\ov \eta}_\eps|^2
	\!+ \tfrac 12 |\CC - {\ov \CC}_\eps|^2.
\end{align*}
The adapted optimization problem $(CP)_\ept^{\rm ad}$
is then defined exactly as in \eqref{CP:adapted} with this new definition for the cost functional.

Arguing as in Section~\ref{ssec:eps0}, 
we \sfw ly infer the corresponding of Lemmas~\ref{LEM:EXISTENCE:ADAPTED}--\ref{LEM:OPT:ADAPTED:EPS}
and Theorem~\ref{THM:CONV:EPS:ADAPTED}, 
whose proofs are omitted since can be reproduced in the same fashion.
These results concern solvability of $\CPad$, necessary conditions for $\CPad$,
and approximation $\CPad\searrow(CP)_\eps$ as ${\tau}\searrow 0$.
Let us just point out that since 
the cost functional is corrected 
also with respect to the state variable,
the forcing term of the corresponding adjoint system
has to be corrected too, with no additional effort though.

\begin{lemma}
\label{LEM:EXISTENCE:ADAPTED:TAU}
Assume {\bf A1--A7} and {\bf C1--C3}.
Then, for every $\tau \in (0,\tau_0)$ and for every 
minimizer $(\bph_\eps, \ov \P_\eps, \ov\chi_\eps, \ov\eta_\eps, \ov \CC_\eps)$
of $(CP)_\eps$, 
the optimization problem $\CPad$ admits a minimizer. 
\end{lemma}

\begin{lemma}
\label{LEM:OPT:ADAPTED:TAU}
Assume {\bf A1--A7} and {\bf C1--C4}, and let
$(\bph_\eps, \ov \P_\eps, \ov\chi_\eps, \ov\eta_\eps, \ov \CC_\eps)$
be a minimizer for $(CP)_\eps$.
For every $\tau\in(0,\tau_0)$,
if $({\ov \P}_\ept,{\ov \chi}_\ept,{{\ov \eta}_\ept,}{\ov \CC}_\ept)$
is an optimal control for $\CPad$, then
the following first-order necessary condition holds
\begin{align}
	&\non
	\intQ (\P - \ov{\P}_\ept \,) \bs_\ept f(\bph_\ept) p_\ept
	- \intQ (\chi - \ov{\chi}_\ept ) \bs_\ept q_\ept
	- \intQ (\CC - \ov{\CC}_\ept ) \bs_\ept f(\bph_\ept) \, r_\ept
	\\ & \quad  \non
	+  (\P - \ov{\P}_\ept \,)\big( \aP(\ov{\P}_\ept-\P_*) + (\ov{\P}_\ept -\ov{\P}_\eps)\big)
	+ (\chi - \ov{\chi}_\ept \,)\big(\achi (\ov{\chi}_\ept-\chi_*) + (\ov{\chi}_\ept-\ov{\chi}_\eps)\big)
	\\ & \quad \non
	+ (\eta - \ov{\eta}_\ept \,)\big(\aeta (\ov{\eta}_\ept-\eta_*) 
	+ (\ov{\eta}_\ept-\ov{\eta}_\eps)\big)\\
	&\quad
	+ (\CC - \ov{\CC} _\ept\,)\big( \aC (\ov{\CC}_\ept-\CC_*) + (\ov{\CC} _\ept-\ov{\CC}_\eps) \big)
	\geq 0 \qquad\text{for every $(\P,\chi,\eta,\CC) \in \Uad$},
	\label{opt:ad:tautozero}
\end{align}
where $(\bph_\ept,\bm_\ept,\bs_\ept)$ and $(p_\ept,q_\ept, r_\ept)$ are
the corresponding unique solutions to the state system \statesys\ and 
the adjoint system \adjsys\ with $\ept>0$
and with respect to the 
coefficients $({\ov \P}_\ept,{\ov \chi}_\ept,{{\ov \eta}_\ept,}{\ov \CC}_\ept)$,
the right-hand side of \eqref{adj:1} 
being modified as $\beta_Q(\bph_\ept-\varphi_Q)+(\ov\varphi_\ept-\bph_\eps)$.
\end{lemma}

\begin{theorem}
\label{THM:CONV:TAU:ADAPTED}
Assume \ref{ass:1}--\ref{ass:7}, \ref{ass:op:1}--\ref{ass:op:4}, and
let $(\bph_\eps, {\ov \P}_\eps,{\ov \chi}_\eps,{{\ov \eta}_\eps,}{\ov \CC}_\eps)$ 
be a minimizer for $\CPeps$.
Then, for every family of optimal controls 
$\{( {\ov \P}_{\eps,\tau},{\ov \chi}_{\eps,\tau},{{\ov \eta}_{\eps,\tau},}{\ov \CC}_{\eps,\tau})\}_\tau$
for $\CPad$,
with corresponding states
$\{(\bph_{\eps,\tau},\bm_{\eps,\tau},\bs_{\eps,\tau})\}_\tau$, 
as $\tau\to0$ it holds that,
for every $\alpha \geq 1$ if $d=2$ and $1\leq \alpha <6$ if $d=3$,
\begin{align*}
	& \non
	\bph_{\eps,\tau} \to \bph_{{\eps}} \quad \hbox{weakly-$^*$ in $\L\infty H \cap \L2 V,$}
	\\ & \hspace{3cm} 
	\hbox{and strongly in $\L2 {\Lx \alpha},$}
	\\ 
	&{\ov \P}_{\eps,\tau}\to {\ov \P}_{{\eps}}, \quad 
	{\ov \chi}_{\eps,\tau}\to {\ov \chi}_{{\eps}},
	\quad
	{{\ov \eta}_{\eps,\tau}\to {\ov \eta}_{{\eps}},
	\quad}
	{\ov \CC}_{\eps,\tau} \to {\ov \CC}_{{\eps}},
	\\
	& 
	\Jad (\bph_{\eps,\tau},{\ov \P}_{\eps,\tau},
	{\ov \chi}_{\eps,\tau},{{\ov \eta}_{\eps,\tau},}
	{\ov \CC}_{\eps,\tau}) \to 
	\J(\bph_\eps, {\ov \P}_\eps, {\ov \chi}_{\eps} {, {\ov \eta}_{\eps}},{\ov \CC}_{\eps}).
\end{align*}
\end{theorem}
\begin{proof}
  The proof is analogous to the one of Theorem~\ref{THM:CONV:EPS:ADAPTED},
  the only difference being that the identification of the limit $\hat\varphi=\ov\ph_\eps$
  follows from the additional correction in the cost functional
  (and not from the uniqueness of the state system with $\tau=0$,
  which is indeed not true). The strong convergence of $\{\bph_\ept\}_\tau$
  is a consequence of the convergences in Theorem~\ref{THM:PREV:ASY:tauzero}
  and the compact inclusion $V\hookrightarrow L^\alpha(\Omega)$ 
  for every $\alpha \geq 1$ if $d=2$ and $1\leq \alpha <6$ if $d=3$.
\end{proof}

\subsubsection{Letting $\tau\to0$ in the adjoint system}
\label{SUB:ADJ:TAUTOZERO}
In this section we study the passage to the limit as $\tau\searrow0$
in the adjoint system \adjsys\ , where the forcing term 
of \eqref{adj:1} is modified as stated in Lemma~\ref{LEM:OPT:ADAPTED:TAU}.

\begin{theorem}
	\label{THM:ADJ:TAUTOZERO}
	Assume \ref{ass:1}--\ref{ass:8} and \ref{ass:op:1}--\ref{ass:op:4}.
	Let the parameters 
	$(\P_\eps,\chi_\eps,\eta_\eps,\CC_\eps)\in\Uad$ and 
	$\{(\P_\ept, \chi_\ept, \eta_\ept, \CC_\ept)\}_\tau\subset\Uad$
	be 
	such that $(\P_\ept, \chi_\ept, \eta_\ept, \CC_\ept)\to (\P_\eps,\chi_\eps,\eta_\eps,\CC_\eps)$
	 as $\tau\to0$.
	Let $(\bph_\eps,\bm_\eps,\bs_\eps)$ be a solution to the state system \statesys\
	with $\tau=0$ and coefficients $(\P_\eps,\chi_\eps,\eta_\eps,\CC_\eps)$
	as given by Theorem~\ref{THM:PREV:ASY:tauzero},
	and let $(\bph_\ept, \bm_\ept,\bs_\ept)$ be the solution to 
	the state system \statesys\
	with $\ept>0$ and coefficients $(\P_\ept, \chi_\ept, \eta_\ept, \CC_\ept)$
	as given by Theorem~\ref{THM:WP:STRONG}.
	Let also $(p_\ept,q_\ept,r_\ept)$ be the unique solution to
	the adjoint system \adjsys\
	with $\ept >0$, coefficients $(\P_\ept, \chi_\ept, \eta_\ept, \CC_\ept)$,
	and forcing term in \eqref{adj:1} given by
	$\beta_Q(\bph_\ept-\ph_Q)+(\bph_\ept-\bph_\eps)$.
	Then, there exist a triplet $(p_\eps,q_\eps,r_\eps)$, with 
	\begin{align*}
	&p_\eps,r_\eps \in \H1 H\cap L^\infty(0,T; V)\cap L^2(0,T; W),
	\qquad q_\eps\in L^2(0,T; H),
	\end{align*}
	such that as $\tau\to0$,
	for every $\alpha \geq 1$ if $d=2$ and $1\leq \alpha <6$ if $d=3$, it holds that
	\begin{align*}
	p_{\eps,\tau} & \to p_\eps && \hbox{weakly-$^*$ in $\H1 H \cap  \L\infty V \cap \L2 W$},
	\\ & && \quad \hbox{strongly in $\C0 {\Lx \alpha} \cap \L2 V$,}
	\\
	q_{\eps,\tau} & \to q_\eps && \hbox{weakly in $\L2 H$},
	\\
	r_{\eps,\tau} & \to r_\eps && \hbox{weakly-$^*$ in $\H1 H \cap \L\infty V \cap \L2 W$},
	\\ & && \quad \hbox{strongly in $\C0 {\Lx \alpha} \cap \L2 V$,}
	\\
	\tau q_{\eps,\tau} & \to 0 && \hbox{strongly in $\H1 {L^1(\Omega)} \cap \L2 H$}.
\end{align*}
	Moreover, $(p_\eps,q_\eps,r_\eps)$ is the unique
	weak solution to the adjoint system \adjsys\ with $\tau =0$
	and coefficients $(\P_\eps,\chi_\eps,\eta_\eps,\CC_\eps)$,
	in the sense that
	\begin{align*}
	& \non
	-\int_\Omega \partial_t p_\eps  v
	+ \iO (a q_\eps - J*q_\eps )v +\eta_\eps\int_\Omega\Delta r_\eps v
	+ \iO  F''(\bph_\eps)q_\eps v
	\\ & \qquad
	+ \iO  \CC_\eps\bs_\eps f'(\bph_\eps) r_\eps v
	- \iO  \P_\eps\bs_\eps f'(\bph_\eps) p_\eps v
	=\,\iO  	\bQ (\bph_\eps-\phQ) v,
	\\	
	& 	-\eps\int_\Omega \partial_t p_\eps  w+
	\iO \nabla p_\eps\cdot \nabla  w
	- \iO q_\eps w	\,=\,	0,
	\\
	& -\iO \dt r_\eps z
	+ \iO \nabla r_\eps \cdot \nabla z
	+ \iO \CC_\eps f(\bph_\eps) r_\eps z
	- \iO \P_\eps f(\bph_\eps)p_\eps z 
	- \iO \ch_\eps q_\eps z 
	\,=0,
\end{align*}
for every $v,w,z  \in V$, almost everywhere in $(0,T)$, and
\[
	p_\eps(T)\,=\, 0, \qquad r_\eps(T)\,=0.
\]
\end{theorem}

\begin{proof}[Proof of Theorem \ref{THM:ADJ:TAUTOZERO}]
We proceed by pointing out some a priori estimates 
on the adjoint variables uniformly with respect to $\tau$,
using again the estimate \eqref{est3_adj} as a starting point.
First of all, by Theorem~\ref{THM:PREV:ASY:tauzero}
we have that $\{F(\bph_\ept)\}_\tau$
is uniformly bounded in $L^\infty(0,T; L^1(\Omega))$,
so that by {\bf A8}
we have that $\{F''(\bph_\ept)\}_\tau$ is uniformly bounded in $L^\infty(0,T; H)$.
Secondly, by the assumption on the kernel $J$,
the Young inequality, and by comparison in equation \eqref{adj:2} we have 
\begin{align*}
  \int_{Q_t^T}(J*q_\ept)q_\ept&\leq\int_t^T\norma{J*q_\ept}_V\norma{q_\ept}_{V^*}\leq
  (a^*+b^*)\int_t^T\norma{q_\ept}\norma{q_\ept}_{V^*}\\
  &\leq\delta\int_{Q_t^T}|q_\ept|^2+
  \frac{(a^*+b^*)^2}{2\delta}\int_t^T\norma{p_\ept}_{V}^2
  +\frac{(a^*+b^*)^2}{2\delta}\eps^2\int_{Q_t^T}|\partial_t p_\ept|^2
\end{align*}
while the last two terms on the right-hand side of \eqref{est3_adj}
can be bounded as
\begin{align*}
&- \intQtT \eta_{\eps,\tau} \Delta r_\ept q_\ept
	- \intQtT \chi_{\eps,\tau} q_\ept (\dt r_\ept+\Delta r_\ept)\\
&\qquad\leq
	\delta\int_{Q_t^T}|q_\ept|^2 + 
	\frac{3(\eta_\ept^2+\chi_\ept^2)}{4\delta}\int_{Q_t^T}|\Delta r_\ept|^2
	+\frac{3\chi_\ept^2}{4\delta}\int_{Q_t^T}|\dt r_\ept|^2.
\end{align*}
Hence, recalling that $\{\ov\sigma_\ept\}_\tau$ is uniformly bounded in $L^2(0,T; V)$,
all the terms in \eqref{est3_adj} can be 
rearranged and treated by the Gronwall lemma,
provided to fix $\delta,\delta'>0$ such that 
\[
  3\delta<C_0\,, \qquad
  \frac{(a^*+b^*)^2}{2\delta}\eps<1,
  \qquad \frac{3(\eta_\ept^2+\chi_\ept^2)}{4\delta} < 1, \qquad
  \delta'<1-\frac{3(\eta_\ept^2+\chi_\ept^2)}{4\delta}.
\]
An easy computation shows that this is possible if and only if
\[
  \max\left\{\frac{(a^*+b^*)^2}{2}\eps,
  \frac{3(\eta_\ept^2+\chi_\ept^2)}{4} \right\}< \frac{C_0}{3}
\]
which is indeed true by {\bf A6} and the fact that 
 $\eta_{\rm max}^2+\chi_{\rm max}^2<\frac49C_0$.
Hence, \eqref{est3_adj} can be closed uniformly in $\tau$, and we obtain
\begin{align*}
	&\norma{p_{\ept}}_{\H1 H \cap \L\infty V}
	+ \tau ^{1/2}\norma{q_{\ept}}_{\L\infty H}
	\\ & \quad 
	+ \norma{q_{\ept}}_{\L2 H}
	+ \norma{r_{\ept}}_{\H1 H \cap \L\infty V\cap \L2 W}
	\leq M_\eps
\end{align*}
for a positive constant $M_\eps$ which may depend on $\eps$ but it is independent of $\tau.$
Then, elliptic regularity theory and \eqref{adj:2} lead us to infer that
\begin{align*}
	\norma{p_\ept}_{\L2 W}
	\leq  M_\eps.
\end{align*}
Moreover, noting that by the H\"older inequality we have that 
\[
  \norma{F''(\bph_\ept)q_\ept}_{L^2(0,T; L^1(\Omega))}\leq M_\eps,
\]
arguing as in the proof of Theorem~\ref{THM:ADJ:EPSTOZERO},
by a comparison argument in \eqref{adj:1} 
we deduce that
\begin{align*}
	\tau \norma{q_\ept}_{\H1 {L^1(\Omega)}} \leq M_\eps.
\end{align*}
Recalling the continuous embedding $W\hookrightarrow L^\infty(\Omega)$,
Banach--Alaoglu theorem and standard compactness results 
allow us to obtain from the above a priori estimates that 
there exist functions $(p_\eps,q_\eps,r_\eps)$ such that as $\tau \to 0$
it holds, along a non-relabeled subsequence,
\begin{align*}
	p_\ept & \to p_\eps && \hbox{weakly-$^*$ in $\H1 H \cap  \L\infty V \cap \L2 W$},
	\\
	q_\ept & \to q_\eps && \hbox{weakly in $\L2 H$},
	\\
	r_\ept & \to r_\eps && \hbox{weakly-$^*$ in $\H1 H \cap \L\infty V \cap \L2 W$},
	\\ & && \quad \hbox{strongly in $\C0 {\Lx \alpha} \cap \L2 V$,}
	\\
	\tau q_\ept & \to 0 && \hbox{weakly in $\H1 {W^*}$ and strongly in $\L2 H$},
\end{align*}
for every $\alpha \geq 1$ if $d=2$ and $1\leq \alpha <6$ if $d=3$.
Arguing as in the proof of Theorem~\ref{THM:ADJ:EPSTOZERO},
we exploit the above convergences to pass to the limit in
the variational formulation of the adjoint system given by
\eqref{adj:vf:1}--\eqref{adj:vf:3} and in the terminal conditions \eqref{adj:vf:4}.
As a by-product, we obtain that the above limits are a weak solution to \adjsys\ with $\tau=0$.
To this end, note that by \eqref{tautozero:1} and 
\eqref{tautozero:strong:1}--\eqref{tautozero:strong:2},
along a non-relabelled subsequence,
for every $\alpha \geq 1$ if $d=2$ and $1\leq \alpha <6$ if $d=3$, we have that
\begin{align*}
	\bph_\ept\to\bph_\eps \quad\text{a.e.~in $Q$}, \qquad
	\bph_\ept &\to \bph_\eps, \quad \bs_\ept \to \bs_\eps \qquad
	 \hbox{strongly in $\L2 {\Lx \alpha}$}.
\end{align*}
Consequently, all terms in \eqref{adj:vf:1}--\eqref{adj:vf:3} and \eqref{adj:vf:4}
pass to the weak limit as $\tau\to0$. The only delicate term to treat, 
as usual, is the one containing $F''$: let us spend a few words on this.
By continuity of $F''$ it follows that, as $\tau \to 0$,
\[
  F''(\bph_\ept)\to F''(\bph_\eps) \quad\text{a.e.~in $Q$}.
\]
Now, since $\{F(\bph_\ept)\}_\tau$ is bounded in $L^\infty(0,T; L^1(\Omega))$,
by {\bf A8} we know that 
$\{F''(\bph_\ept)\}_\tau$ is bounded in $L^\infty(0,T; H)$.
Furthermore, the boundedness of $\{\varphi_\ept\}_\tau$ in $L^2(0,T; L^6(\Omega))$
and again {\bf A8} ensure also that 
$\{F''(\bph_\ept)\}_\tau$ is uniformly bounded in $L^1(0,T; L^3(\Omega))$.
For any $\vartheta\in(0,1)$, setting 
$\beta_\vartheta\in(2,3)$ such that $\frac1{\beta_\vartheta}
:=\frac\vartheta2+\frac{1-\vartheta}{3}$,
by interpolation we have that 
\[
  \norma{F''(\bph_\ept)}_{\beta_\vartheta}\leq
  \norma{F''(\bph_\ept)}^\vartheta
  \norma{F''(\bph_\ept)}_{3}^{1-\vartheta} \quad\text{a.e.~in $(0,T)$},
\]
from which it follows that 
\[
  \norma{F''(\bph_\ept)}_{L^{\frac1{1-\vartheta}}(0,T; L^{\beta_\vartheta}(\Omega))}\leq M_\eps.
\]
In particular, 
there exists $\ov\vartheta\in (0,1)$ such that 
$\ov\beta:=\beta_{\ov\vartheta}=\frac1{1-\ov\vartheta}\in(2,3)$:
an easy computation yields $\ov\vartheta=\frac47$ and $\ov\beta=\frac73$.
This implies that
\[
  \norma{F''(\bph_\ept)}_{L^{7/3}(Q)}\leq M_\eps.
\]
By the Severini-Egorov theorem we infer that,
for all $\beta\in[1,\frac73)$,
\[
  F''(\bph_\ept)\to F''(\bph_\eps) \quad\text{weakly in $L^{7/3}(Q)$ and
  	strongly in $L^\beta(Q)$}.
\]
In particular, since $\frac73>2$, this implies that 
\[
  F''(\bph_\ept)q_\ept\to F''(\bph_\eps)q_\eps \quad\text{weakly in $L^1(Q)$.}
\]
Since $W\hookrightarrow L^\infty(\Omega)$, this allows to pass to the 
limit as $\tau\to0$ in \eqref{adj:vf:1} for every test function $v\in W$.
Since $F''(\bph_\eps)\in L^1(0,T; L^3(\Omega))$,
at the limit we have that $F''(\bph_\eps)q_\eps\in L^{6/5}(\Omega)\hookrightarrow V^*$
almost everywhere in $(0,T)$, and the variational formulation holds 
also for all $v\in V$ by the density of $W$ in $V$.

Finally, by linearity of the system, the same estimates yield 
also uniqueness of $(p_\eps,q_\eps,r_\eps)$, and the convergences 
holds along the entire sequence $\tau\to0$, as desired.
\end{proof}

\subsubsection{Letting $\tau \to 0$ in the optimality condition}
Lastly, we argue as in  Theorem~\ref{THM:AD:OPT:EPS}
to pass to the limit as $ \tau \to 0$
to establish the necessary conditions for optimality of $\CPeps$.

\begin{theorem}
\label{THM:AD:OPT:TAU}
Assume \ref{ass:1}--\ref{ass:8} and \ref{ass:op:1}--\ref{ass:op:4}.
Then, every minimizer 
$(\bph_\eps, {\ov \P}_\eps, {\ov \chi}_\eps, {\ov \eta}_\eps, {\ov \CC}_\eps)$ of $\CPeps$
necessarily satisfies
\begin{align*}
	& \intQ (\P - \ov{\P}_\eps \,) \bs_\eps f(\bph_\eps) p_\eps
	- \intQ (\chi - \ov{\chi}_\eps ) \bs_\eps q_\eps
	- \intQ (\CC - \ov{\CC}_\eps ) \bs_\eps f(\bph_\eps) \, r_\eps
	\\ & \quad
	+ \aP (\ov{\P}_\eps-\P_*)(\P- \ov{\P}_\eps \,)
	+ \achi (\ov{\chi}_\eps -\chi_*)(\chi - \ov{\chi}_\eps \,)
	\\ & \quad
	+ \aeta(\ov{\eta}_\eps -\eta_*)(\eta - \ov{\eta}_\eps \,)
	+ \aC (\ov{\CC}_\eps-\CC_*)(\CC - \ov{\CC}_\eps \,)
	\geq 0
	\qquad \hbox{for every $(\P,\chi,\eta,\CC) \in \Uad$,}
\end{align*}
where $(\bph_\eps, \bm_\eps,\bs_\eps)$ is a solution to 
the state system \statesys\ and
and $(p_\eps,q_\eps,r_\eps)$ is the unique weak solution to
the adjoint system \adjsys\
with $\tau=0$ and coefficients 
$({\ov \P}_\eps, {\ov \chi}_\eps,{\ov \eta}_\eps, {\ov \CC}_\eps)$,
in the sense of Theorems \ref{THM:PREV:ASY:tauzero} and 
Theorem \ref{THM:ADJ:TAUTOZERO}, respectively.
\end{theorem}
\begin{proof}
  The proof is analogous to the one of Theorem~\ref{THM:AD:OPT:EPS},
  using Lemma~\ref{LEM:EXISTENCE:ADAPTED:TAU} and
  \ref{LEM:OPT:ADAPTED:TAU}, and
  Theorems~\ref{THM:PREV:ASY:tauzero}, 
  \ref{THM:CONV:TAU:ADAPTED}, and \ref{THM:ADJ:TAUTOZERO} instead.
\end{proof}

\subsection{The optimization problem $(\ov{CP})$}
\label{ssec:ept0}
In this last section we deal with 
the optimization problem $(\ov{CP})$, by letting 
$\ept\to0$ jointly. Since most of the ideas have already been 
explained and motivated in detail in the previous 
Sections~\ref{ssec:eps0}--\ref{ssec:tau0}, we proceed here
more quickly, avoiding technical details for brevity.
Throughout the whole Section~\ref{ssec:ept0}, we 
assume the following framework:
\begin{align*}
  &\eta_{\rm max}=\alpha_{\rm max}=0, \qquad\beta_\Omega\varphi_\Omega\in V,
  \qquad
  \text{\eqref{init_tau0}--\eqref{bound_init_tau0}},\\
  &\chi_{\rm max} < \sqrt{c_a}, \qquad
   (\chi_{\rm max}+\eta_{\rm max}+4c_a\chi_{\rm max})^2<8c_aC_0,
   \qquad\eta_{\rm max}^2+\chi_{\rm max}^2<\frac{4}{9}C_0.
\end{align*}
As in Section~\ref{ssec:eps0}, since $\eta_{\rm max}=0$, 
we shall consider $\Uad$ as a subset of $\erre^3$ instead, 
and write $(\P,\chi,\CC)\in\Uad$ for the generic admissible control.

As usual, existence of optimal controls for $(\ov{CP})$
 is given in the following first result.
\begin{theorem}
Assume {\bf A1--A8} and {\bf C1--C3}.
Then, the optimization problem $(\ov{CP})$
admits a solution.
\end{theorem}
\begin{proof}
This result follows directly by adapting the direct method used in
the proof of Theorem~\ref{THM:EXISTECE:MIN}, taking into account the 
compactness of $\Uad$ and the convergence pointed out in Theorem~\ref{THM:PREV:ASY:epstautozero}.
\end{proof}

Now, we investigate necessary conditions for optimality.
First of all, note that for every admissible control $(\P,\chi,\CC)$,
the state system \statesys\ with $\eps=\tau=0$
admits a unique solution by Theorem~\ref{THM:PREV:ASY:epstautozero}.
Consequently, when introducing the adapted cost functional,
by contrast with Section~\ref{ssec:tau0},
it is not necessary here to use a perturbation with respect to the phase variable.

Nevertheless, looking at the final condition \eqref{adj:5}
and taking formally $\eps=\tau=0$, we immediately see that 
$p_\ept(T)=0$ for all $\eps\in(0,\eps_0)$ 
while at the limit $p(T)=\beta_\Omega(\ov\varphi(T)-\varphi_\Omega)$.
This immediately suggests that if $\beta_\Omega>0$,
we cannot pass to the joint limit $\ept\to0$
in the adjoint problem \adjsys\ as it is. 
At an intuitive level for the moment, 
the limit adjoint problem \adjsys\ with $\eps=\tau=0$
is still well-posed also when $\beta_\Omega>0$.
These heuristic considerations
suggest that the right assumption is to keep a generic $\beta_\Omega\geq0$,
but to modify the final condition for $p_\ept$ at the approximate level in a smart way,
in order to recover the compatibility condition 
$p_\ept(T)=\beta_\Omega(\ov\varphi(T)-\varphi_\Omega)$ also when $\ept>0$.
In order to do this, we introduce a correction in the adapted cost functional,
depending on a suitable combination of
the terminal values of both the variables $\varphi$ and $\mu$.

For a given optimal control $(\ov \P,\ov\chi,\ov \CC)$ of $(\ov{CP})$
and for all $\eps\in(0,\eps_0)$ and $\tau\in(0,\tau_0)$,
the idea is to set 
\begin{align*}
	\Jad(\ph,\mu, \P, \chi, \CC) &: = \J(\ph, \P, \chi, \CC)
	+\left(\eps\mu(T), \beta_\Omega(\varphi(T)-\varphi_\Omega)\right)_H  \\
	&\qquad+ \tfrac 12 |\P -  \ov \P|^2
	+ \tfrac 12 |\chi - \ov \chi|^2
	+ \tfrac 12 |\CC - \ov \CC|^2
\end{align*}
and define the adapted optimization problem $(CP)_\ept^{\rm ad}$ as in \eqref{CP:adapted}.

It is clear that the optimality condition for $\CPoo$ 
follows similar lines of the previous sections,
by firstly obtaining the approximating sequence of optimal controls of minimizers of
$\CPad$ and then pass to the limit as $\ept \to 0$.
The major difference is the nature of the correction in the adapted cost 
functional, which yields a correction in the terminal values of the adapted adjoint system
at $\ept>0$: for this reason, we recall such corrected adjoint system explicitly
in Lemma~\ref{LEM:OPT:ADAPTED:EPT} below.
The proof of first-order conditions for optimality at the level $\ept>0$
follows, {\em mutatis mutandis}, the proof of Theorem~\ref{THM:OPT:FIRST},
and is omitted for brevity.

The following results concern existence of optimal controls
and first-order necessary conditions for $(CP)_\ept^{\rm ad}$, and the 
convergence $(CP)_\ept^{\rm ad}\searrow(\ov{CP})$ as $\ept\to0$.
\begin{lemma}
\label{LEM:EXISTENCE:ADAPTED:EPT}
Assume \ref{ass:1}--\ref{ass:8} and \ref{ass:op:1}--\ref{ass:op:4}. 
Then, for every optimal control $(\ov \P, \ov\chi, \ov \CC)$ for $(\ov{CP})$,
for every $\eps \in (0,\eps_0)$ and $\tau\in(0,\tau_0)$ 
the optimization problem $\CPad$ admits a solution. 
\end{lemma}

\begin{lemma}
\label{LEM:OPT:ADAPTED:EPT}
Assume \ref{ass:1}--\ref{ass:8} and \ref{ass:op:1}--\ref{ass:op:4},
and let $(\ov \P, \ov\chi, \ov \CC)$ be an optimal control  for $(\ov{CP})$.
For every $\eps\in(0,\eps_0)$ and $\tau\in(0,\tau_0)$,
if $({\ov \P}_\ept,{\ov \chi}_\ept,{\ov \CC}_\ept)$ is an optimal control of $\CPad$, 
then the following first-order necessary condition holds
\begin{align}
	&\non \intQ (\P - \ov{\P}_\ept \,) \bs_\ept f(\bph_\ept) p_\ept
	- \intQ (\chi - \ov{\chi}_\ept ) \bs_\ept q_\ept
	- \intQ (\CC - \ov{\CC}_\ept ) \bs_\ept f(\bph_\ept) \, r_\ept
	\\ &\quad \non
	+ (\P - \ov{\P}_\ept \,)\big( \aP (\ov{\P}_\ept-\P_*) + (\ov{\P}_\ept -\ov{\P})\big)
	+(\chi - \ov{\chi}_\ept \,)\big( \achi (\ov{\chi}_\ept-\chi_*) + (\ov{\chi}_\ept-\ov{\chi})\big)
	\\ & \quad
	+ (\CC - \ov{\CC} _\ept\,)\big( \aC (\ov{\CC}_\ept-\CC_*) + (\ov{\CC} _\ept-\ov{\CC}) \big)
	\geq 0
	\qquad
	\hbox{for every $(\P,\chi,\CC) \in \Uad$,}
	\label{opt:ad:epttozero}
\end{align}
where $(\bph_\ept,\bm_\ept,\bs_\ept)$ is the unique solution to \statesys\
with $\ept>0$ and parameters $({\ov \P}_\ept,{\ov \chi}_\ept,{\ov \CC}_\ept)$, 
and $(p_\ept,q_\ept, r_\ept)$ is the unique solution to the following 
adapted adjoint system
\begin{align}
	& \non\label{adj:1_ad}
	-\dt (p_\ept +\tau q_\ept)
	+a q_\ept - J*q_\ept
	+ F''(\bph_\ept)q_\ept
	\\ & \qquad
	+ \ov\CC\bs f'(\bph_\ept) r_\ept
	- (\ov\P\bs_\ept -\A)f'(\bph_\ept) p_\ept \,=\, 	\bQ (\bph_\ept-\phQ)
	\qquad&&\text{in } Q,
	\\	
	\label{adj:2_ad}
	& 	-\eps \dt p_\ept	- \Delta p_\ept 	- q_\ept 	\,=\,	0
	\qquad&&\text{in } Q,
	\\
	\label{adj:3_ad}
	& -\dt r_\ept
	- \Delta r_\ept
	+ (\B + \ov\CC f(\bph_\ept)) r_\ept
	- \ov\P f(\bph_\ept)p_\ept
	- \ov\ch q_\ept
	\,=0
	\qquad&&\text{in } Q,
	\\
	&\dn p_\ept\,=\,\dn r_\ept \,=\,0
	\qquad&&\text{on } \Sigma,
	\label{adj:4_ad}
	\\ 
	& \non\eps p_\ept(T) \,=\,\eps\beta_\Omega(\ov\varphi_\ept(T)-\varphi_\Omega),
	\\
	&\qquad 
	(p_\ept+\tau q_\ept)(T)\,=\, \bO (\bph_\ept(T)-\phO) + 
	\eps\beta_\Omega\ov\mu_\ept(T), \qquad
	r_\ept(T)\,=0
	\qquad&&\text{in } \Omega.
	\label{adj:5_ad}
\end{align}
\Accorpa\adjsysad {adj:1_ad} {adj:5_ad}
\end{lemma}

\begin{theorem}
\label{THM:CONV:EPT}
Assume \ref{ass:1}--\ref{ass:8} and \ref{ass:op:1}--\ref{ass:op:4}.
Let $(\ov \P, \ov \chi, \ov \CC)$ be an optimal control for $\CPoo$,
with corresponding state $(\bph, \bm, \bs)$.
Then, for every family 
$\{( {\ov \P}_{\ept},{\ov \chi}_{\ept},{\ov \CC}_{\ept})\}_{\ept}$ of optimal controls
for $(CP)_\ept^{\rm ad}$, 
with corresponding states
$\{(\bph_{\ept},\bm_{\ept},\bs_{\ept})\}_\ept$,
as $\ept\to0$ it holds that, 
for every $\alpha \geq 1$ if $d=2$ and $1\leq \alpha <6$ if $d=3$,
\begin{align*}
	& \non
	\bph_{\ept} \to \bph \quad \hbox{weakly$^*$ in $\L\infty H \cap \L2 V$,}
	\\ & \hspace{3cm}
	\hbox{strongly in $\C0 \Vp \cap \L2 {\Lx \alpha}$},
	\\ 
	&{\ov \P}_{\ept} \to {\ov \P}_\tau, \quad 
	{\ov \chi}_{\ept}\to {\ov \chi}_\tau,
	\quad
	{\ov \CC}_{\ept} \to {\ov \CC}_\tau,
	\\
	& 
	\Jad (\bph_{\ept},{\ov \P}_{\ept},
	{\ov \chi}_{\ept},{\ov \CC}_{\ept}) \to \J(\bph, {\ov \P}, {\ov \chi},{\ov \CC}).
\end{align*}
\end{theorem}
\begin{proof}
  The proof is analogous to Theorem~\ref{THM:CONV:EPS:ADAPTED},
  by using the convergences of Theorem~\ref{THM:PREV:ASY:epstautozero}.
\end{proof}

\subsubsection{Letting $(\eps,\tau)\to(0,0)$ in the adjoint system}
We focus here on the passage to the limit as $(\eps,\tau)\to(0,0)$
in the adjoint system \adjsys.

\begin{theorem}
\label{THM:ADJ:EPSTAUTOZERO}
	Assume \ref{ass:1}--\ref{ass:8} and \ref{ass:op:1}--\ref{ass:op:4}.
	Let $(\P,\chi,\CC)\in\Uad$, $\{(\P_\ept,\chi_\ept,\CC_\ept)\}_\ept\subset\Uad$
	be such that $(\P_\ept,\chi_\ept,\CC_\ept)\to(\P,\chi,\CC)$ as $\ept\to0$.
	Let $(\bph,\bm,\bs)$ and $(\bph_\ept,\bm_\ept,\bs_\ept)$ be the unique solutions
	to the state system \statesys\ in the cases $\eps=\tau=0$ with coefficients
	$(\P,\chi,\CC)$, and $\eps\in(0,\eps_0)$ and $\tau\in(0,\tau_0)$
	with coefficients $(\P_\ept,\chi_\ept,\CC_\ept)$, 
	as given by Theorems~\ref{THM:PREV:ASY:epstautozero} and 
	\ref{THM:WP:STRONG}, respectively.
	Let also $(p_\ept,q_\ept,r_\ept)$ be the unique solution to the adapted adjoint system 
	\adjsysad\ with $\eps\in(0,\eps_0)$, $\tau\in(0,\tau_0)$,
	and coefficients $(\P_\ept,\chi_\ept,\CC_\ept)$, as given by Theorem~\ref{THM:ADJ:ept}.
	Then, there exists a triplet $(p,q,r)$, with
	\[
	p\in H^1(0,T; V^*)\cap L^2(0,T; W),\quad
	q\in L^2(0,T; H),\quad
	r\in \H1 H \cap \L\infty V \cap \L2 W,
	\]
	such that, 
	for every $\alpha \geq 1$ if $d=2$ and $1\leq \alpha <6$ if $d=3$, 
	along any double sequence 
	\[
	(\eps,\tau)\to(0,0) \qquad \text{such that}\qquad\limsup_{(\ept)\to(0,0)}\frac\eps\tau<+\infty,
	\]
	it holds
	\begin{align*}
	p_{\ept} & \to p && \hbox{weakly in $\H1 \Vp \cap \L2 W$},
	\\
	q_{\ept} & \to q && \hbox{weakly in $\L2 H$},
	\\
	r_{\ept} & \to r && \hbox{weakly-$^*$ in $\H1 H \cap \L\infty V \cap \L2 W$},
	\\ & && \quad \hbox{strongly in $\C0 {\Lx \alpha} \cap \L2 H$,}
	\\
	\eps p_{\ept} & \to 0
	&& \hbox{strongly in $\H1 H \cap \L\infty \Vp$},
	\\
	\tau  q_{\ept} & \to 0
	&& \hbox{strongly in $\L\infty V$}.
	\end{align*}
	Moreover, $(p,q,r)$ is the unique weak solution to the adjoint system
	\adjsys\ with $\eps =\tau =0$ and coefficients $(\P,\eta, \CC)$, in the sense that
		\begin{align*}
	& \non
	-\<\dt p , v>_{V}
	+ \iO (a q - J*q +F''(\bph) )v
	+ \iO  \CC\bs f'(\bph) rv
	- \iO  \P\bs f'(\bph) pv
	=\iO  \bQ (\bph-\phQ) v,
	\\	
	& 	
	\iO \nabla p\cdot \nabla  w
	- \iO q w	\,=\,	0,
	\\
	& -\iO \dt r z
	+ \iO \nabla r \cdot \nabla z
	+ \iO \CC f(\bph) r z
	- \iO \P f(\bph)p z 
	- \iO \ch q z 
	\,=0,
\end{align*}
for every $v,w,z  \in V$, almost everywhere in $(0,T)$, and
\[
	p(T)\,=\, \bO (\bph(T)-\phO), \qquad r(T)\,=0.
\]
\end{theorem}
\begin{proof}[Proof of Theorem \ref{THM:ADJ:EPSTAUTOZERO}]
We perform on the adapted adjoint system \adjsysad\ the same first estimate
of the proof of Theorem~\ref{THM:ADJ:ept}, getting
\begin{align*}
	& 
	\frac {\tau} 2 \IO2 {q_\ept}
	+  C_0 \int_{Q_t^T}|q_\ept|^2
	+ \eps \int_{Q_t^T}|\dt p_\ept|^2
	+ \frac \eps2 \IO2 {p_\ept}
	+ \frac 12  \IO2 {\nabla p_\ept}
	+ \int_{Q_t^T}|\nabla p_\ept|^2
	\\ & \qquad
	+\frac {\B+1}2 \norma{r_\ept(t)}^2
	+ \IO2 {\nabla r_\ept} + \B\int_{Q_t^T}|\nabla r_\ept|^2
	+ \int_{Q_t^T}|\dt r_\ept|^2
	+  \int_{Q_t^T}|\Delta r_\ept|^2
	\\ & \leq\frac {\tau} 2 \|q_\ept(T)\|^2
	+\frac\eps 2\|p_\ept(T)\|^2	
	+\frac12\|\nabla p_\ept(T)\|^2
	+ \intQtT  \bQ(\bph_\ept-\phQ)q_\ept
	+ \intQtT (J*q_\ept)q_\ept
	\\ & \qquad
	- \intQtT  \ov\CC_\ept \bs f'(\bph_\ept)r_\ept q_\ept
	+\intQtT  (\ov\P_\ept\bs_\ept-\A) f'(\bph_\ept)p_\ept q_\ept
	+\intQtT q_\ept p_\ept
	\\ & \qquad
	+ \intQtT \ov\CC_\ept f(\bph_\ept) r_\ept (\dt r_\ept+\Delta r_\ept)
	- \intQtT \ov\P_\ept f(\bph_\ept)p_\ept (\dt r_\ept+\Delta r_\ept)
	\\ & \qquad
	- \intQtT \ov\chi_\ept q_\ept (\dt r_\ept+\Delta r_\ept)
	- \intQtT r_\ept \dt r_\ept.
\end{align*} 
We show only how to bound the first three terms on the right-hand side, 
as all the other terms on the right-hand side can be treated exactly in the same way 
as in Section~\ref{ssec:est_adj} and in the proof of Theorem~\ref{THM:ADJ:TAUTOZERO},
using Theorem~\ref{THM:PREV:ASY:epstautozero}.
To this end,
taking into account the modified
terminal conditions \eqref{adj:5_ad} we have 
\[
  p_\ept(T)=\beta_\Omega(\ov\varphi_\ept(T)-\varphi_\Omega) \in V, \qquad
  q_\ept(T)= \frac\eps\tau\beta_\Omega\ov\mu_\ept(T).
\]
It follows that 
\begin{align*}
 &  \frac {\tau} 2 \|q_\ept(T)\|^2
  	+\frac\eps 2\|p_\ept(T)\|^2
	+\frac12\|\nabla p_\ept(T)\|^2
	\\ & \quad =
	\frac{\beta_\Omega^2}2\frac{\eps^2}{\tau}\|\ov\mu_\ept(T)\|^2
	  +\frac\eps 2\|\beta_\Omega(\ov\varphi_\ept(T)-\varphi_\Omega)\|^2
	+\frac12\|\nabla\beta_\Omega(\ov\varphi_\ept(T)-\varphi_\Omega)\|^2.
\end{align*}
Now, by Theorem~\ref{THM:PREV:ASY:epstautozero}
we have that $\{\eps^{1/2}\ov\mu_\ept\}_\ept$ is uniformly bounded in $C^0([0,T]; H)$: hence,
thanks also to the regularity 
$\beta_\Omega\varphi_\Omega\in V$, we deduce that there exists 
$M>0$, independent of $\ept$, such that 
\[
  \frac {\tau} 2 \|q_\ept(T)\|^2
  +\frac\eps 2\|p_\ept(T)\|^2
  +\frac12\|\nabla p_\ept(T)\|^2 \leq M(1+ \eps + \tfrac\eps\tau).
\]
Consequently, the scaling $\limsup \frac\eps\tau<+\infty$ on $(\eps,\tau)$ yields that
\[
  \frac {\tau} 2 \|q_\ept(T)\|^2
    +\frac\eps 2\|p_\ept(T)\|^2
  +  \frac12\|\nabla p_\ept(T)\|^2 \leq M.
\]
As we anticipated above, all the remaining terms on the right-hand side
can be treated exactly as in Section~\ref{ssec:est_adj} and in the proof of Theorem~\ref{THM:ADJ:TAUTOZERO}, so that 
we infer that there exists a positive constant $M$,
independent of both $\eps$ and $\tau$, such that
\begin{align*}
	&\eps^{1/2} \norma{p_{\ept}}_{\H1 H}
	+ \norma{p_{\ept}}_{\L\infty V}
	+ \tau ^{1/2}\norma{q_{\ept}}_{\L\infty H}
	+ \norma{q_{\ept}}_{\L2 H}
	\\ & \qquad 
	+ \norma{r_{\ept}}_{\H1 H \cap \L\infty V\cap L^2(0,T; W)}
	\leq M.
\end{align*}
Moreover, elliptic regularity theory and \eqref{adj:2}--\eqref{adj:3} leads to 
\begin{align*}
	\norma{p_\ept}_{\L2 W}
	\leq  M,
\end{align*}
while by comparison in \eqref{adj:1}, 
as in the proof of Theorem~\ref{THM:ADJ:TAUTOZERO}, we infer that
\begin{align*}
	\norma{p_\ept+\tau q_\ept}_{\H1 {L^1(\Omega)}}
	\leq M.
\end{align*}
By the usual (weak) compactness criteria, we infer 
the existence of functions $(p, q, r)$ such that as $\ept \to 0$
it holds, 
for every $\alpha \geq 1$ if $d=2$ and $1\leq \alpha <6$ if $d=3$,
and along a non-relabelled subsequence,
\begin{align*}
	p_\ept & \to p && \hbox{weakly in $\L2 W$},
	\\
	q_\ept& \to q && \hbox{weakly in $\L2 H$},
	\\
	p_\ept+\tau q_\ept& \to p && \hbox{weakly in $\H1 {W^*}$},
	\\
	r_\ept & \to r && \hbox{weakly-$^*$ in $\H1 H \cap \L\infty V \cap \L2 W$},
	\\ & && \quad \hbox{strongly in $\C0 {\Lx \alpha} \cap \L2 V$,}
	\\
	\eps p_\ept & \to 0
	&& \hbox{strongly in $\H1 H \cap \L2 W$},
	\\
	\tau q_\ept & \to 0
	&& \hbox{strongly in $\L\infty H$}.
\end{align*}
Moreover, since by Theorem~\ref{THM:PREV:ASY:epstautozero}
we have that, for every $\alpha \geq 1$ if $d=2$ and $1\leq \alpha <6$ if $d=3$,
\begin{align*}
	\bph_\ept \to \bph \quad \hbox{strongly in $\C0 \Vp \cap \L2 {\Lx \alpha}$},
\end{align*}
we can pass to the limit as $\ept\to0$ in
the variational formulation \eqref{adj:vf:1}--\eqref{adj:vf:4},
treating the term with $F''$
as in the proof of Theorem~\ref{THM:ADJ:TAUTOZERO}, and conclude.
The uniqueness of the weak solution $(p,q,r)$ follows from linearity 
and the estimates already performed.
\end{proof}

\subsubsection{Letting $(\eps,\tau)\to(0,0)$ in the optimality condition}
Here, we conclude the asymptotic analysis by letting $\ept\to0$ in
the optimality condition for $(CP)_\ept^{\rm ad}$, and
proving the corresponding necessary conditions for $(\ov{CP})$.
\begin{theorem}
\label{THM:AD:OPT:EPT}
Assume \ref{ass:1}--\ref{ass:8} and \ref{ass:op:1}--\ref{ass:op:4}.
Then, every optimal control $({\ov \P}, {\ov \chi}, {\ov \CC})$ of $\CPoo$
necessarily satisfies
\begin{align*}
	& \intQ (\P - \ov{\P} \,) \bs f(\bph)  p
	- \intQ (\chi - \ov{\chi} ) \bs q
	- \intQ (\CC - \ov{\CC} ) \bs f(\bph) \,  r
	\\ & \quad
	+ \aP (\ov{\P}-\P_*)(\P- \ov{\P}\,)
	+ \achi (\ov{\chi}-\chi_*)(\chi - \ov{\chi} \,)
	+ \aC (\ov{\CC}-\CC_*)(\CC - \ov{\CC} \,)
	\geq 0
	\\ & \quad\text{for every $(\P,\chi,\CC) \in \Uad$},
\end{align*}
where $(\bph, \bm,\bs)$
and $(p,q, r)$ are the unique solutions to \statesys\ and \adjsys\ with 
$\eps=\tau =0$
in the sense of Theorems~\ref{THM:PREV:ASY:epstautozero} and 
\ref{THM:ADJ:EPSTAUTOZERO}, respectively.
\end{theorem}

\section*{Acknowledgments}
\noindent
This research was supported by the Italian Ministry of Education, University and Research
(MIUR): Dipartimenti di Eccellenza Program (2018–2022) – Dept. of Mathematics “F. Casorati”, University of Pavia.
In addition, this research has been performed in the framework of the project Fondazione Cariplo-Regione Lombardia  MEGAsTAR
``Matema\-tica d'Eccellenza in biologia ed ingegneria come acceleratore
di una nuova strateGia per l'ATtRattivit\`a dell'ateneo pavese''. The present paper
also benefits from the support of the GNAMPA (Gruppo Nazionale per l'Analisi Matematica, la Probabilit\`a e le loro Applicazioni)
of INdAM (Istituto Nazionale di Alta Matematica).
LS is funded by the Austrian Science Fund (FWF) through the Lise-Meitner 
grant M 2876.

\vspace{3truemm}
\footnotesize
\bibliographystyle{abbrv}
\def\cprime{$'$}

\end{document}